\documentclass[a4paper,12pt,reqno]{amsart}
\usepackage{amsmath,graphicx,graphics,amssymb}

\usepackage[font=scriptsize]{caption}

\newtheorem{theo}{Theorem}[section]
\newtheorem{lem}[theo]{Lemma}

\newtheorem{cor}[theo]{Corollary}

\numberwithin{equation}{section}

\newcommand{\M}{\operatorname{M}}

\newcommand{\de}{\operatorname{d}}
\newcommand{\q}{\operatorname{q}}
\newcommand{\e}{\operatorname{e}}
\newcommand{\f}{\operatorname{f}}

\newcommand{\h}{\operatorname{h}}

\newcommand{\Z}{\mathbb{Z}}

\newcommand{\ve}{\operatorname{v}}

\newmuskip\pFqskip
\mathchardef\pFcomma=\mathcode`, % keep a copy of the comma

\begin{document}

\title[Round Aztec windows and a symmetry of the correlation of diagonal slits]{Round Aztec windows, a dual of the Aztec diamond theorem and a curious symmetry of the correlation of diagonal slits}

%\title{Cruciform regions and a conjecture of Di Francesco}
%Symmetries of shamrocks III: 
%%Graphical condensation with free boundary and 
%The symmetric self complementary case}
%\title{The number of lozenge tilings of a hexagon with dents on the boundary}

\author{Mihai Ciucu}
\address{Department of Mathematics, Indiana University, Bloomington, Indiana 47405}
\email{mciucu@iu.edu}
\thanks{Research supported in part by Simons Foundation Collaboration Grant 710477}

\begin{abstract} Fairly shortly after the publication of the Aztec diamond theorem of Elkies, Kuperberg, Larsen and Propp in 1992, interest arose in finding the number of domino tilings of an Aztec diamond with an ``Aztec window,'' i.e.\ a hole in the shape of a smaller Aztec diamond at its center. Several intriguing patterns were discovered for the number of tilings of such regions, but the numbers themselves were not ``round'' --- they didn't seem to be given by a simple product formula. In this paper we consider a very closely related shape of holes (namely, odd Aztec rectangles), and prove that a large variety of regions obtained from Aztec rectangles by making such holes in them possess the sought-after property that the number of their domino tilings is given by a simple product formula. We find the same to be true for certain symmetric cruciform regions. We also consider graphs obtained from a toroidal Aztec diamond by making such holes in them, and prove a simple formula that governs the way the number of their perfect matchings changes under a natural evolution of the holes. This yields in particular a natural dual of the Aztec diamond theorem. Some implications for the correlation of such holes are also presented, including an unexpected symmetry for the correlation of diagonal slits on the square grid.
\end{abstract}

\maketitle

%
%\begin{figure}[h]
%\centerline{
%\hfill
%{\includegraphics[width=0.90\textwidth]{proof8a.eps}}
%\hfill
%}
%\vskip-0.1in
%\caption{Schematic representation of the bijection proving \eqref{eba}. Shifting along path containing $a$ matches the partition classes according to the pattern}
%\vskip-0.1in
%\label{fba}
%\end{figure}
%

\section{Introduction}

%Some background/history:

%EKLP

The subject of this article is the enumeration of domino tilings of some families of regions on the square lattice. This is part of a vast area, comprising also tilings of other lattices, and more generally perfect matchings of graphs. For an overview the reader is referred to Propp's survey article \cite{ProppTilings}.

The Aztec diamond of order $n$ is the region $AD_n$ on the square grid obtained by stacking $2n$ unit-height horizontal strips of lengths $2,4,6,\dotsc,2n,2n,2n-2,\dotsc,2$ on top of each other so that they have a common vertical symmetry axis (the outside boundary of the region on the left in Figure \ref{fia} traces out $AD_{18}$). These regions were introduced by Elkies, Kuperberg, Larsen and Propp in their 1992 paper \cite{EKLP}, where they proved that the number of their domino tilings is given by the very simple formula\footnote{ For a region $R$ on the square lattice we denote by $\M(R)$ the number of tilings of $R$.}
\begin{equation}
\M(AD_n)=2^{n(n+1)/2}.
\label{eia}
\end{equation}
%
%Aztec window in the center (domino list 1996, Propp's paper 1996/1999): some interesting patterns and properties, but no product formula (or other explicit formula)
Fairly shortly after that, interest arose in finding the number of domino tilings of an Aztec diamond with an “Aztec window,” i.e. a hole in the shape of a smaller Aztec diamond at its center (see Figure \ref{fia} for an example). More precisely, Problem 14 in Propp's list of open problems in \cite{ProppList} (whose first form was published in 1996) presents several intriguing patterns that the number of tilings of such regions seem to satisfy (these patterns were conjectured based on data), but the numbers themselves are not ``round'' --- the presence of relatively large primes in their factorizations suggests that they are not given by a simple product formula.

%By contrast, it turns out that making {\it odd} holes leads to simple product formulas, in a much more general situation. Namely, consider the {\it Aztec rectangle} region $AR_{m,n}$, and assume remove from along 

The main point of this paper is to show that if we change the shape of the window to being an {\it odd} Aztec diamond (and more generally, an odd Aztec rectangle), we obtain families of regions that do have this sought-after feature --- the number of their domino tilings is given by simple product formulas.

More precisely, the {\it odd Aztec diamond} of order $n$ is the region\footnote{ In the previous literature this region is usually denoted by $OD_n$; due to its ubiquitous appearance, we denote it in this paper by the simpler $O_n$. This region was also referred to by James Propp (see \cite[Problem~27]{ProppList}) as the ``fool's diamond,'' because, as one readily sees, it has no domino tilings. Nevertheless, it plays an essential role in this paper!} $O_n$ obtained by stacking $2n+1$ unit-height horizontal strips of lengths $1,3,5,\dotsc,2n-1,2n+1,2n-1,\dotsc,1$ on top of each other so that they have a common vertical symmetry axis (in the picture on the right in  Figure \ref{fia}, the boundary of the hole traces out the odd Aztec diamond $O_2$).

One readily sees that if the unit squares of $\Z^2$ are colored in chessboard fashion, then $O_n$ contains $2n+1$ more unit squares of the majority color. Since $AD_n$ is {\it balanced}, i.e.\ it has the same number of black and white unit squares, making a hole in $AD_n$ in the shape of an $O_k$ will produce a region which is not balanced, hence not tileable by dominos. A natural way to fix this is to start with an {\it Aztec rectangle} $AD_{m,n}$ (in the picture on the right in Figure \ref{fia} the outside boundary traces out $AD_{23,18}$); in it, the difference between the number of unit squares of the two colors is readily seen to be $|m-n|$; by choosing appropriate values for the parameters $m$ and $n$, this can be adjusted so that the Aztec rectangle with the odd Aztec diamond hole
%\footnote{ Or more generally, with an odd Aztec rectangle hole; see the picture on the right in Figure \ref{fbaaa}.}
is balanced (see the picture on the right in Figure \ref{fia} for an example).

\begin{figure}[t]
\centerline{
%\hfill
{\includegraphics[width=0.45\textwidth]{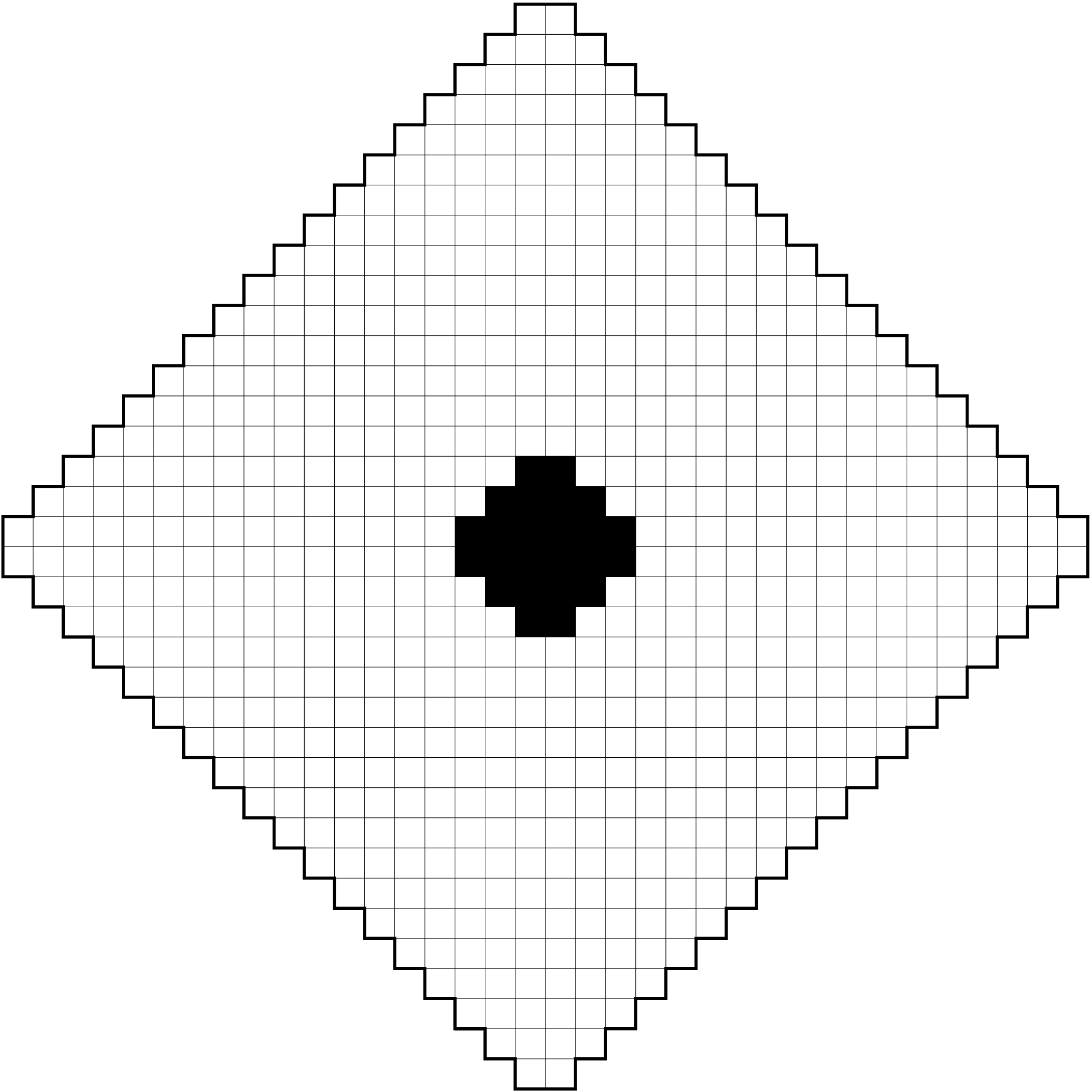}}
\hfill
{\includegraphics[width=0.512\textwidth]{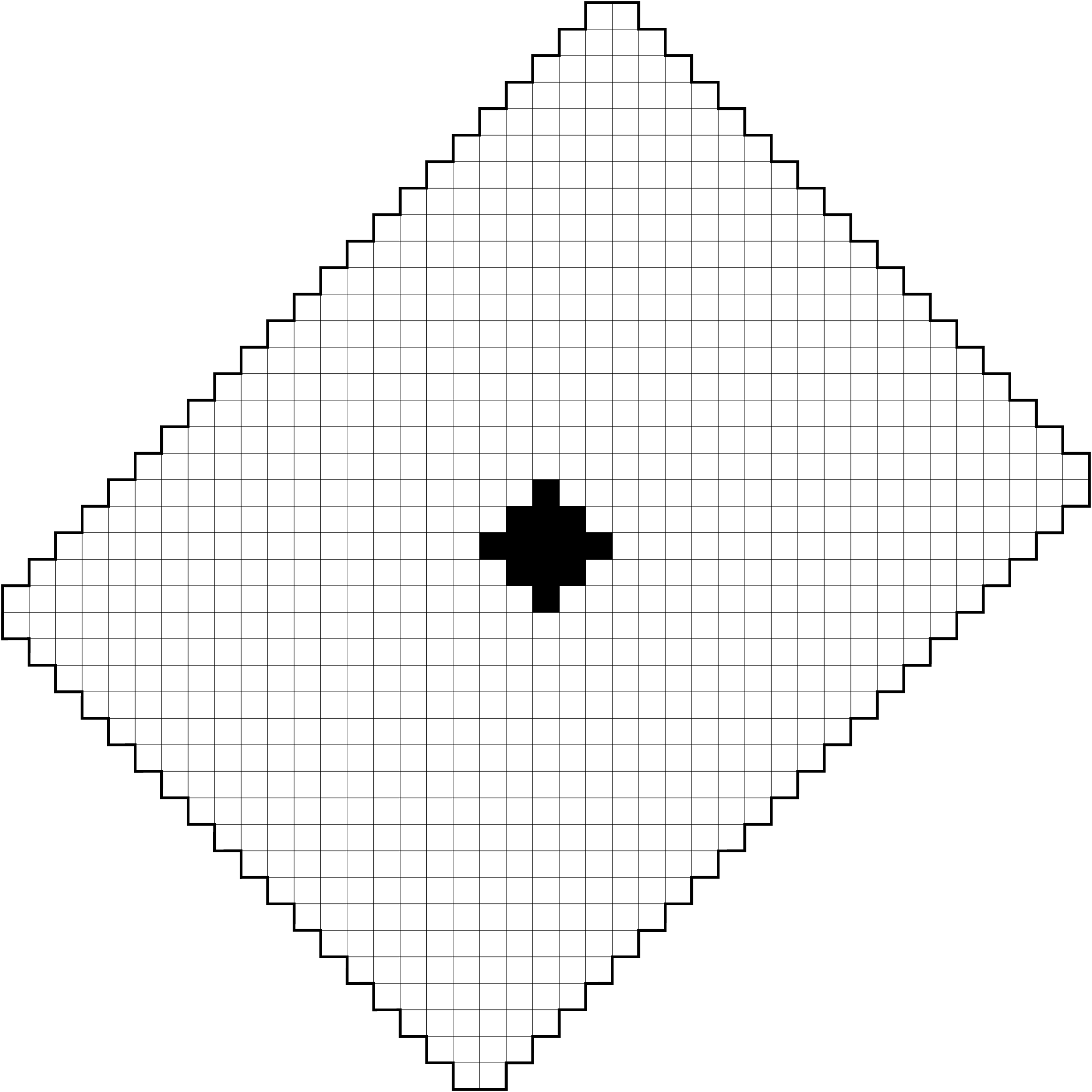}}
%\hfill
}
\vskip0.1in
\caption{ {\it Left.} The Aztec diamond $AD_{18}$ with an $AD_3$-shaped Aztec window at its center. It has $2^{87}\cdot3^2\cdot11^2\cdot101^2\cdot131^2\cdot6961^2$ domino tilings. {\it Right.} The Aztec rectangle $AD_{23,18}$ with an $O_2$-shaped odd Aztec window along its symmetry axis. It has $2^{118}\cdot3^4\cdot11^6\cdot13^8\cdot17^4\cdot19^2$ domino tilings.}
\vskip-0.1in
\label{fia}
\end{figure}

We will show more generally that any collection of same-height odd Aztec rectangle holes placed along the symmetry axis of an Aztec rectangle region produces a region whose number of domino tilings is given by a simple product formula. This is presented in Section 3 (see Figures \ref{fba}--\ref{fbc} and Theorems \ref{tba}--\ref{tbd}). Our results are phrased in terms of the number of perfect matchings of Aztec rectangle graphs from which a collinear set of vertices has been removed (either from the horizontal symmetry axis or from the grid diagonal just below this symmetry axis), for which product formulas were found in our earlier work \cite{FT} and by Krattenthaler in \cite{Kratt}; for convenience, these are recalled in Section 2.

In Section 4 we consider general collections of odd Aztec rectangular holes (i.e., not necessarily collinear or of the same height), and encase them in {\it toroidal} Aztec rectangles (obtained from Aztec rectangle graphs by identifying corresponding vertices on opposite sides of the boundary; see Figure \ref{fca} for an illustration). We prove that if these holes evolve in a certain natural way, the number of perfect matchings of the corresponding toroidal graphs with holes changes just by a multiplicative factor, which is a power of 2 (see Theorem \ref{tca}). Several consequences are also deduced, including a natural dual of Elkies, Kuperberg, Larsen and Propp's Aztec diamond theorem (see Corollary \ref{tcf}) and
a surprising symmetry of the correlation of diagonal slits, which holds already at the level of finite size correlations (see Corollary \ref{tcd} and Remark 6).

Our proofs are presented in Sections 5 and 6. They are based on the complementation theorem for perfect matchings of cellular graphs we obtained in \cite{CT}. 

%(more generally, odd Azted rectangle holes) leads to simple product formulas, in a much more general situation. Namely, consider the {\it Aztec rectangle} region $AR_{m,n}$, and assume remove from along 

%\newpage

\begin{figure}[t]
\centerline{
\hfill
{\includegraphics[width=0.25\textwidth]{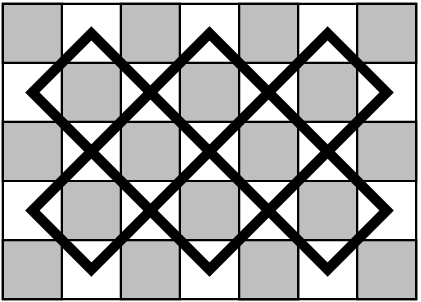}}
\hfill
{\includegraphics[width=0.25\textwidth]{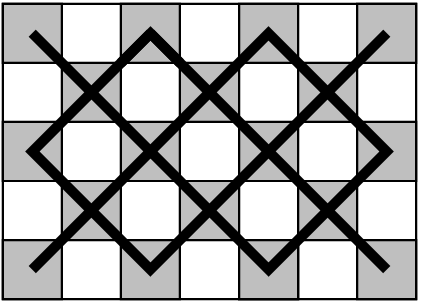}}
\hfill
}
\vskip0.1in
\caption{ Two graphs constructed from a $(2m+1)\times(2n+1)$ chessboard with black corners: the Aztec rectangle graph $AR_{m,n}$ (left) and the odd Aztec rectangle graph $O_{m,n}$ (right), shown here for $m=2$, $n=3$.}
\vskip-0.1in
\label{fbaaa}
\end{figure}

\section{Four families of graphs}

Given a $(2m+1)\times(2n+1)$ chessboard with black corners, the {\it Aztec rectangle graph $AR_{m,n}$} is the graph whose vertex set consists of the white squares, with two vertices connected by an edge precisely if the corresponding squares share a corner. The {\it odd Aztec rectangle graph} $O_{m,n}$ is defined analogously, using the black squares (see Figure \ref{fbaaa} for an example of each). The lattice regions whose planar duals are these graphs are called {\it Aztec rectangle} (resp., {\it odd Aztec rectangle}) {\it regions}; see Figure \ref{fba} for some examples (the outer boundary encloses the Aztec rectangle region $AD_{16,28}$, while the regions shown in black are, from left to right, the odd Aztec rectangles $O_{2,2}$, $O_{2,0}$ and $O_{2,1}$). Clearly, the domino tilings of a region on the square lattice can be identified with the perfect matchings of its planar dual graph. We will find it convenient to use both points of view, as the definition of our objects is more natural as regions, but in our proofs we regard them as graphs.

\begin{figure}[t]
\centerline{
\hfill
{\includegraphics[width=0.99\textwidth]{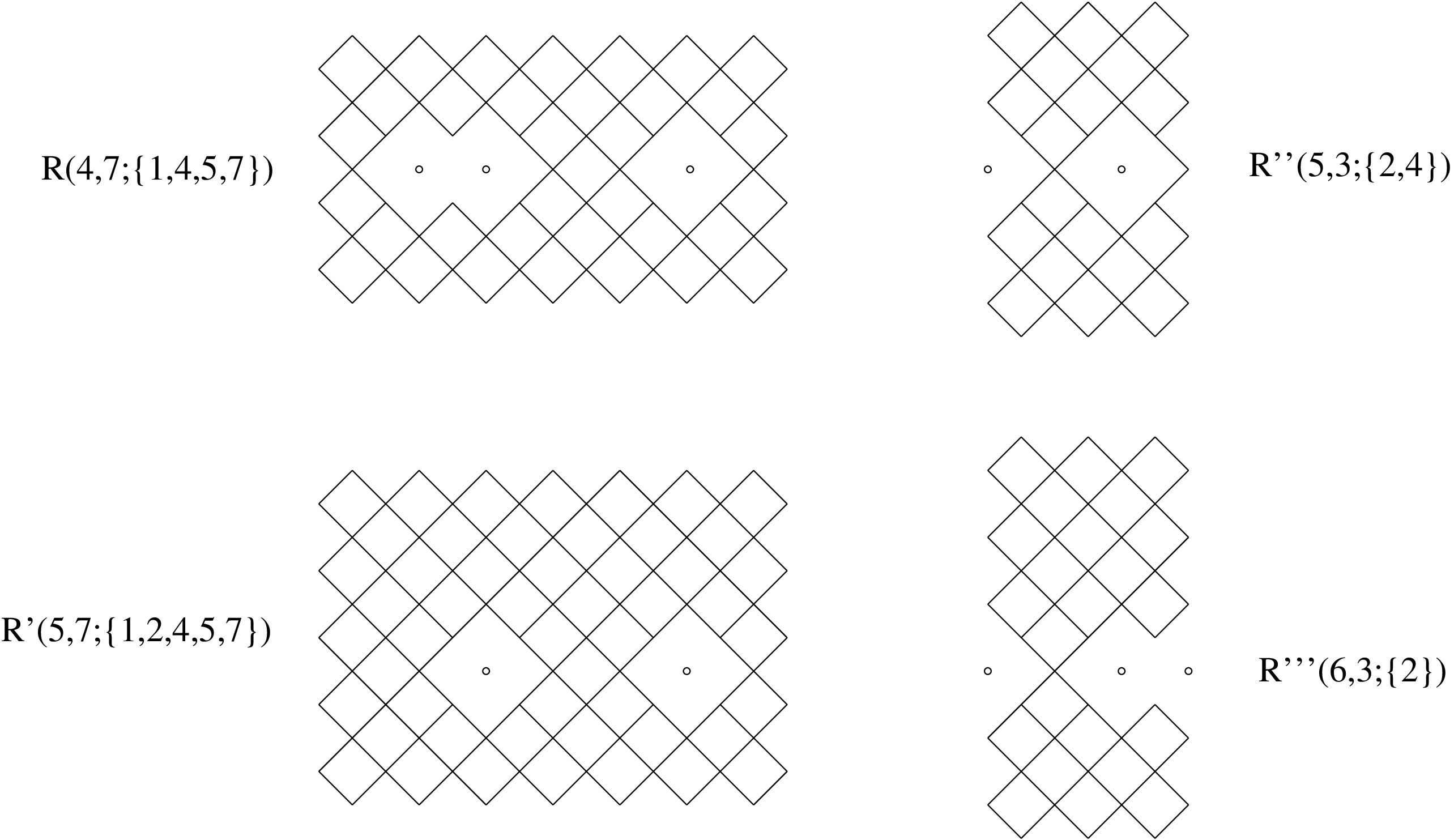}}
\hfill
}
\vskip0.1in
\caption{ Four families of graphs obtained from Aztec rectangles $AR_{m,n}$ by deleting some vertices from the horizontal symmetry axis, or from the row of vertices just below it.
\newline
{\it Left: The case $m<n$ $($height $<$ width$)$.} An example with $m$ even (the graph $R(4,7;\{1,4,5,7\})$, top left), and one with $m$ odd (the graph $R'(5,7;\{1,2,4,5,7\})$, bottom left).
\newline
{\it Right: The case $m>n$ $($height $>$ width$)$.} An example with $m$ odd (the graph $R''(5,3;\{2,4\})$, top right), and one with $m$ even (the graph $R'''(6,3;\{2\})$, bottom right).}
\vskip-0.1in
\label{fbaa}
\end{figure}

Our results in the next section are expressed in terms of the number of perfect matchings of certain graphs obtained from Aztec rectangles by deleting some vertices from their horizontal symmetry axis, or from the row of vertices just below it. There are four such families, which we describe next.

Let $AR_{m,n}$ be the Aztec rectangle graph of height $m$ and width $n$, and color the vertices in its bipartition classes black and white so that the topmost vertices are white. One readily sees that for $m<n$, $AR_{m,n}$ has an excess of $n-m$ white vertices, while for $m>n$ there are $m-n$ more black vertices than white ones. As any perfect matching contains the same number of black and white vertices, we restore the balance by deleting the required number of vertices of the majority color. More precisely, all the deleted vertices will be chosen from along the horizontal symmetry axis $\ell$, if the common color of the vertices on $\ell$ is the majority color (this happens when $m<n$ and $m$ is even, and when $m>n$ and $m$ is odd); otherwise we choose them from the row of vertices just below $\ell$ (these have opposite color to the vertices on $\ell$, so the balancing works out this way).

The detailed definitions are as follows.
Let $m$ and $n$ be positive integers with $m<n$ and~$m$ even.
Then $AR_{m,n}$ has $n$ white vertices on $\ell$, and since $n-m$ of them need to be deleted, $m$ are left.
Label the vertices on $\ell$ from left to right by consecutive integers starting with 1.
Given any subset $T\subset[n]$ with $m$ elements, denote by $R(m,n;T)$ the graph obtained from the Aztec rectangle $AR_{m,n}$ by deleting all the vertices on $\ell$ whose labels are not in $T$; an example is shown on the top left in Figure \ref{fbaa}.

For $m<n$ and $m$ odd, the right color of vertices to delete in order to restore balance is not the one found on $\ell$, but along the line $\ell'$ obtained by translating $\ell$ one unit southeast. Therefore, in this case we label the vertices of  $AR_{m,n}$ that are on $\ell'$ from left to right by $1,2,\dotsc,n$, and for a subset $T\subset[n]$ with $m$ elements we denote by $R'(m,n;T)$ the graph obtained from $AR_{m,n}$ by deleting all the vertices on $\ell'$ whose labels are not in $T$ (see the picture on the bottom left in Figure \ref{fbaa} for an example).

In the case $m>n$ the regions are defined analogously, but the pairing of the parities of $m$ with which of $\ell$ or $\ell'$ to choose to delete vertices from switches. Namely, in this case it is for $m$ odd that the vertices on $\ell$ have the majority color. Therefore, if $m>n$ and $m$ is odd, we delete vertices from $\ell$ to restore the balance. Since in this case $AR_{m,n}$ has $n+1$ vertices on $\ell$, after deleting $m-n$ of them we are left with $2n+1-m$. 
Thus, given $m>n$, $m$ odd, and a subset $T\subset[n+1]$ with $2n+1-m$ elements, we denote by $R''(m,n;T)$ the graph obtained from the Aztec rectangle $AR_{m,n}$ by deleting all the vertices on $\ell$ whose labels are not in $T$ (an example is shown on the top right in Figure \ref{fbaa}). The fourth family corresponds to the case when  $m>n$ and $m$ is even: for such $m$ and $n$,
given a subset $T\subset[n+1]$ with $2n+1-m$ elements, we denote by $R'''(m,n;T)$ the graph obtained from  $AR_{m,n}$ by deleting all the vertices on $\ell'$ whose labels are not in $T$ (see the picture on the bottom right in Figure \ref{fbaa} for an example). 

A product formula for the number of perfect matchings of the graph $R(m,n;T)$ was given in our earlier work \cite{FT} (see Theorem 4.1 there). The perfect matchings of the $R'$, $R''$ and $R'''$ families were enumerated by Krattenthaler in \cite{Kratt} (see Theorems 8, 9 and 10 there); the formula for $\M(R'(m,n;T))$ was also found independently by Helfgott and Gessel \cite[Proposition 8]{HG}. Combinatorial proofs of \cite[Theorem 9 and 10]{Kratt} are presented in \cite[Theorems 7.1 and 7.2]{nesy}. A combinatorial proof of \cite[Theorem 8]{Kratt} follows as a direct consequence of \cite[Theorem 1.2 and equation (4.4)]{FT}.
For ease of reference, we state these formulas below.

\parindent=0pt

\begin{theo}
\label{tbaa}
$(${\rm a}$)$ Let $T=\{t_1,\dotsc,t_m\}\subset[n]$ have its elements listed in increasing order. 

For $m<n$ and $m$ even, we have
\begin{equation}
\M(R(m,n;T))=\frac{2^{m(m+4)/4}}{(0!\,1!\,\cdots\,(m/2-1)!)^2}
\prod_{1\leq i<j\leq m/2}(t_{2j-1}-t_{2i-1})(t_{2j}-t_{2i}),
\label{ebax}
\end{equation}
while for $m<n$ and $m$ odd, 
\begin{align}
&\M(R'(m,n;T))=\frac{2^{(m^2+4m-1)/4}}{\prod_{i=1}^{(m-1)/2}(i-1)!\prod_{i=1}^{(m+1)/2}(i-1)!}
\nonumber
\\[10pt]  
&\ \ \ \ \ \ \ \ \ \ \ \ \ \ \ \ \ \ \ \ \ \ \ \  \ \ \ \ \ \ \ \  
\times
\prod_{1\leq i<j\leq(m-1)/2}(t_{2j}-t_{2i})
\prod_{1\leq i<j\leq(m+1)/2}(t_{2j-1}-t_{2i-1}).
\label{ebay}
\end{align}

$(${\rm b}$)$ Let $T=\{t_1,\dotsc,t_{2n+1-m}\}\subset[n+1]$ have its elements listed in increasing order.

For $m>n$ and $m$ odd, we have
\begin{align}
  &\M(R''(m,n;T))=2^{(m^2-2m+1)/4+n}
\frac{\prod_{i=(m+3)/2}^{n+1}(i-1)!^2}
{\prod_{i=1}^{2n+1-m}(t_i-1)!\,(n+1-t_i)!}
\nonumber
\\[10pt]  
&\ \ \ \ \ \ \ \ \ \ \ \ \ \ \ \ \ \ \ \ \ \ \ \  \ \ \ \ \ \ \ \  
\times
\prod_{1\leq i<j\leq n-(m-1)/2}(t_{2j}-t_{2i})
\prod_{1\leq i<j\leq n-(m-1)/2}(t_{2j-1}-t_{2i-1}),
\label{ebaz}
\end{align}
while for $m>n$ and $m$ even, 
\begin{align}
  &\M(R'''(m,n;T))=2^{(m^2-2m)/4+n}
\prod_{i=(m+2)/2}^{n+1}(i-1)!
\prod_{i=(m+4)/2}^{n+1}(i-1)!
\nonumber
\\[10pt]  
&\ \ \ \ \ \ \ \ \ \ \ \ \ \ \ \ \ \ \ \ \ \ \ \  \ \ \ \ \ \ \ \  
\times
\frac
{\prod_{1\leq i<j\leq n-m/2}(t_{2j}-t_{2i})
 \prod_{1\leq i<j\leq n-m/2+1}(t_{2j-1}-t_{2i-1})}
{\prod_{i=1}^{2n+1-m}(t_i-1)!\,(n+1-t_i)!}.
\label{ebaw}
\end{align}

\end{theo}

\parindent=15pt

\section{Aztec rectangles with odd Aztec windows}

The regions introduced in this section are obtained from Aztec rectangle regions $AR_{m,n}$ by making in them holes (or windows) in the shape of odd Aztec rectangles of some common height, lined up horizontally. The details of their definition --- and the geometric look of the resulting regions --- depend on whether $m<n$ or $m>n$, and whether $m$ is even or odd. This gives rise to four families of regions.

Suppose $m$ is even. Consider integers $a_1,\dotsc,a_s\geq0$. Set $a=\sum_{i=1}^s(a_i+1)$, and draw the Aztec rectangle region $AR_{m,m+a}$ in the plane so that its long symmetry axis $\ell$ is horizontal. Label its unit squares on $\ell$ by $1,2,\dotsc,m+a$, from left to right (see Figure \ref{fba}); regard these as coordinates on $\ell$, viewed as the number line. We call the boundary of this $AR_{m,m+a}$ the {\it reference frame} (or simply the frame; it is shown in green in Figure \ref{fba}).

Let $k\leq\min(a_1,\dotsc,a_s)$ be a positive integer, and start with a larger Aztec rectangle region $AR_{m+k,m+k+a}$ centered around the frame (this is traced by the outer boundary in Figure \ref{fba}). 
As discussed in the third paragraph of Section 2, if the unit squares in this region are colored chessboard style so that the top unit squares are white, there is an excess of $a$ white unit squares. We restore the balance as follows.
For each $i=1,\dotsc,s$, make a hole (or window) in the shape of $O_{k,a_i-k}$, placed so that its center has coordinate $c_i$ on $\ell$ --- if $O_{k,a_i-k}$ has an odd number of unit squares along its horizontal symmetry axis (which happens iff $a_i-k$ is even), then $c_i$ will be an integer; otherwise, $c_i$ will be $1/2$ plus an integer.
Call the resulting balanced\footnote{As we mentioned, the number of white squares minus the number of black squares in the region $AR_{m+k,m+k+a}$ equals $a$. On the other hand, for the region $O_{k,a_i-k}$ this difference is $\pm(a_i+1)$, depending on whether the top unit squares of $O_{k,a_i-k}$ are white or black (the plus sign corresponding to white). It is not hard to check that, since $m$ is even, placing $O_{k,a_i-k}$ symmetrically across $\ell$ makes this difference equal to  $a_i+1$ for all $k$, so that jointly the $O_{k,a_i-k}$'s will have $a=\sum_i(a_1+1)$ more white squares than black ones.}
 region an {\it Aztec rectangle with odd windows}, and denote it by $AR_{m+k,m+k+a}\setminus O_{k,a_1-k}^{(c_1)}\cup\cdots\cup O_{k,a_s-k}^{(c_s)}$; an example is shown in Figure \ref{fba}.

\begin{figure}[t]
\centerline{
\hfill
{\includegraphics[width=0.75\textwidth]{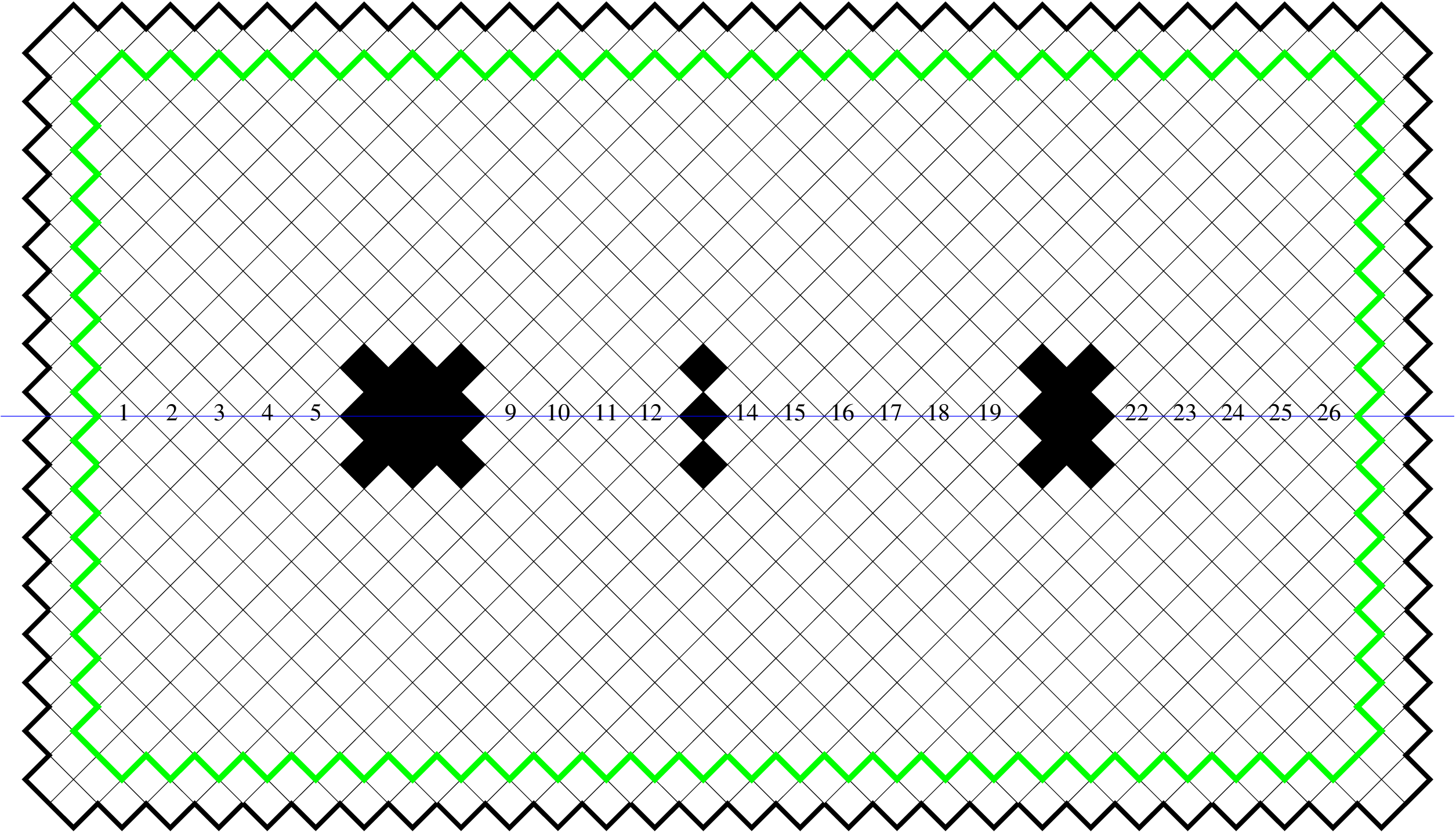}}
\hfill
}
\vskip0.1in
\caption{ The case {\it height} $<$ {\it width} and the region is symmetric (equivalently, the frame height is even). In this example, the frame (shown in green) is $AR_{m,m+a}$ with $m=14$, $a=12$, and the region $AR_{m+k,m+k+a}\setminus O_{k,a_1-k}^{(c_1)}\cup\cdots\cup O_{k,a_s-k}^{(c_s)}$ (whose outer boundary is outlined in black) has the $s=3$ windows $O_{2,2}$, $O_{2,0}$ and $O_{2,1}$ (so $k=2$, $a_1=4$, $a_2=2$, $a_3=3$), placed so that their centers $c_1,\dotsc,c_s$ are $c_1=7$, $c_2=13$, $c_3=20\frac12$ (by definition, $a=\sum_{i=1}^s(a_i+1)$).}
\vskip-0.1in
\label{fba}
\end{figure}

%We recall from our earlier work \cite{FT} (see \cite[Theorem 4.1]{FT}) that for $m$ even, $m\leq n$ and $S=\{s_1,\dotsc,s_{n-m}\}\subset[n]$, if $R(m,n;S)$ is the graph obtained from the Aztec rectangle graph by deleting the $s_1$-st, $s_2$-nd$,\,\dotsc,$\,$s_{n-m}$-th vertices (from left to right) from its horizontal symmetry axis, then the number of perfect matchings of $R(m,n;S)$ is given by
%%
%\begin{equation}
%\M(R(m,n;S))=\frac{2^{m(m+4)/4}}{(0!\,1!\,\cdots\,(m/2-1)!)^2}
%\prod_{1\leq i<j\leq m/2}(t_{2j-1}-t_{2i-1})(t_{2j}-t_{2i}),
%\label{eba}
%\end{equation}
%%
%where $[n]\setminus S=\{t_1,\dotsc,t_m\}$, $t_1<\cdots<t_m$.

Our first main result gives a simple product formula for the number of domino tilings of this family of Aztec rectangles with odd windows (height $<$ width, region is symmetric) in terms of formula \eqref{ebax}.
  
\begin{theo}
\label{tba}
Let $m$ be even, and let $k$ be a positive integer with $k\leq\underset{i}{\min}\, a_i$. Then we have
\begin{align}
&
\M\left(AR_{m+k,m+k+a}\setminus O_{k,a_1-k}^{(c_1)}\cup\cdots\cup O_{k,a_s-k}^{(c_s)}\right)
\nonumber
\\
&\ \ \ \ \ \ \ \ \ \ \ \ \ \ \ \ \ \ \ \ \ \ \ \ \ \ \ \ \ \ \ \ \ \ \ \
=
2^{{k+1\choose2}(s+1)+km}
\M\left(AR_{m,m+a}\setminus O_{0,a_1}^{(c_1)}\cup\cdots\cup O_{0,a_s}^{(c_s)}\right)
\nonumber
\\
&\ \ \ \ \ \ \ \ \ \ \ \ \ \ \ \ \ \ \ \ \ \ \ \ \ \ \ \ \ \ \ \ \ \ \ \
=2^{{k+1\choose2}(s+1)+km}
\M\left(R(m,m+a;A_1^{(c_1)}\cup\cdots\cup A_s^{(c_s)})\right),
\label{ebb}
\end{align}
where $A_i^{(c_i)}$ is the set of $a_i+1$ consecutive integers\footnote{ By construction, the graph $R(m,m+a;A_1^{(c_1)}\cup\cdots\cup A_s^{(c_s)})$ is balanced (i.e., has the same number of black and white vertices when its vertices are colored in a chessboard fashion) if and only if the $A_i^{(c_i)}$'s are mutually disjoint; otherwise it has no perfect matching (as each edge in a perfect matching has one white and one black endpoint).  We will assume therefore, for this case as well as for the remaining three cases, that this condition holds.} centered at $c_i$, and the graph $R(m,n;T)$ is defined in Section $2$.
\end{theo}

Note that the second equality follows by definition, since the odd Aztec rectangle region $O_{0,a_i}^{(c_i)}$ consists of $a_i+1$ consecutive unit squares, whose labels form a run of $a_i$ consecutive integers centered at~$c_i$.

The way to use Theorem \ref{tba} is as follows. Given an Aztec rectangle region with odd windows that belongs to this family (height $<$ width, windows symmetric about $\ell$), read off first the value of $k$ --- it is one unit less than the common height of the odd Aztec windows. Then proceed inwards from the outer boundary, at each step tracing the boundary of an Aztec rectangular region with both dimensions decremented by one unit. This yields at the $k$th step the reference frame --- which is the Aztec rectangle $AR_{m,m+a}$. Label the unit squares on the horizontal symmetry axis from left to right by consecutive integers, starting with 1 for the unit square just inside the frame. Interpret the $R$-graph on the third line in \eqref{ebb} as referring to this labeling of the vertices on the horizontal symmetry axis of the graph $AR_{m,m+a}$, and use formula \eqref{ebax}.

%The choice of the dimensions of the outer boundary of the Aztec rectangles with odd windows defined above was made so that the resulting region with holes is {\it balanced} --- i.e., it has the same number of black and white unit squares in a chessboard coloring of the grid. The condition that $m$ is even was needed in order for this balancing to be possible by making odd Aztec rectangular windows symmetrically along $\ell$. Indeed, for $m<n$, in the chessboard coloring of the Aztec rectangle region $AR_{m,n}$ in which the topmost vertices are white, there are $n-m$ surplus of white unit squares; one readily sees that making an odd Aztec window in $AR_{m,n}$ symmetrically about the latter's long symmetry axis $\ell$ reduces this surplus only if $m$ is even.

Things change in the above discussion if $m$ is odd, in that the balancing of the region $AR_{m+k,m+k+a}$ cannot be achieved by placing the windows $O_{k,a_i-k}$ symmetrically about the horizontal symmetry axis $\ell$; however, the balancing {\it is} achieved if instead we place them so that they are symmetric about $\ell'$, the translation of $\ell$ one unit southeast.

\begin{figure}[t]
\centerline{
\hfill
{\includegraphics[width=0.75\textwidth]{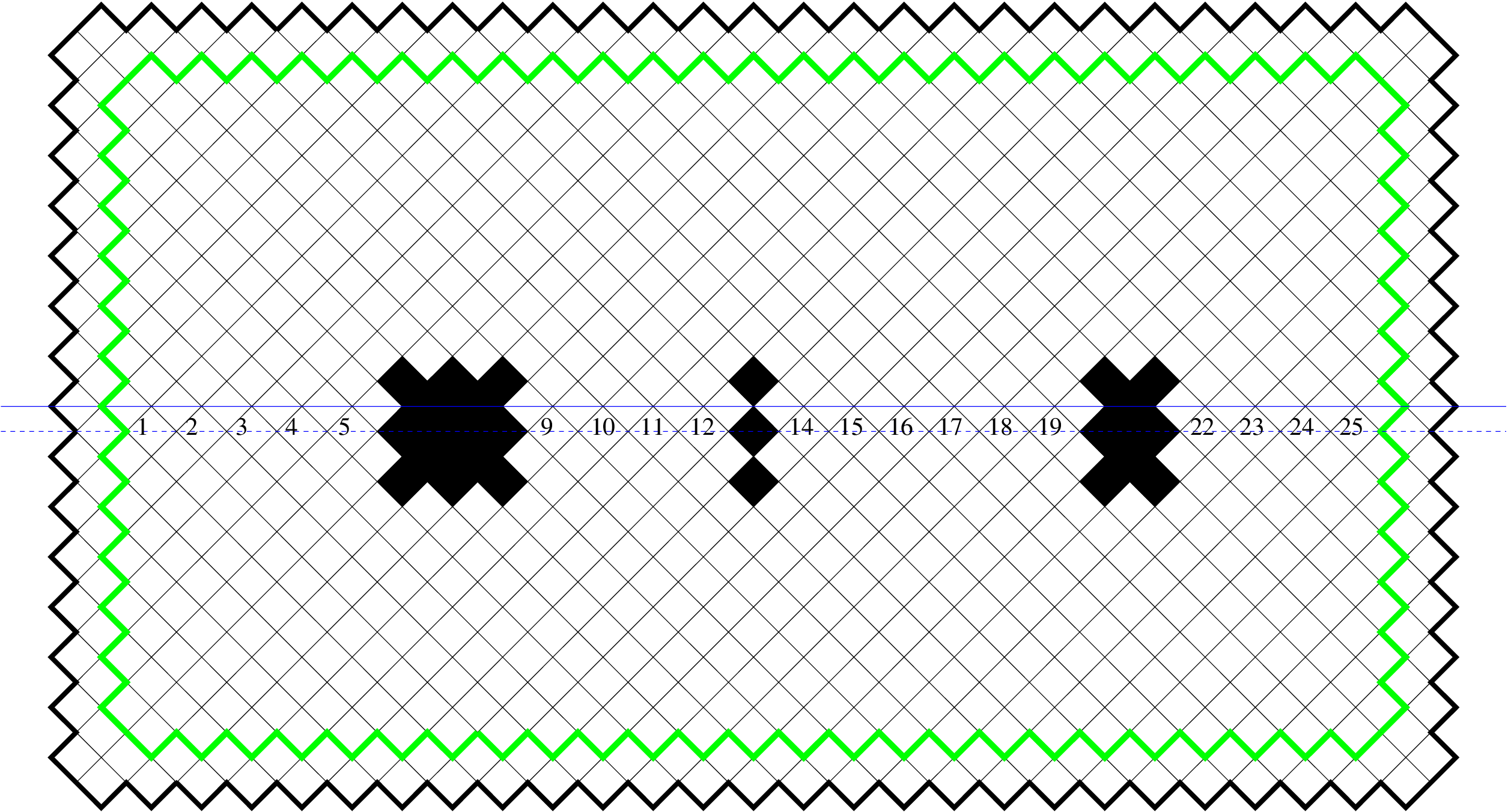}}
\hfill
}
\vskip0.1in
\caption{ The case {\it height} $<$ {\it width} and the region is nearly symmetric (equivalently, the frame height is odd). In this example, the frame (shown in green) is $AR_{m,m+a}$ with $m=13$, $a=12$, and, except for $m$, all the parameters in the region $AR'_{m+k,m+k+a}\setminus O_{k,a_1-k}^{(c_1)}\cup\cdots\cup O_{k,a_s-k}^{(c_s)}$ displayed here are the same as in Figure \ref{fba}.}
\vskip-0.1in
\label{fbb}
\end{figure}

Because of this, for $m$ odd, we define the Aztec rectangles with odd windows by making a small modification in the above definition. Namely, label by $1,2,\dotsc$ the unit squares {\it on $\ell'$} (and not on the symmetry axis $\ell$), starting with the one just inside the frame, and place the odd windows symmetrically about $\ell'$. Apart from this change, the rest of the definition is exactly the same as for $m$ even. We denote the resulting region by\footnote{ Strictly speaking, one can tell just from the indices of an Aztec rectangle with odd windows if the $m$-parameter is even or odd, as $m=(m+k)-k$, $m+k$ is the height of the outer boundary, and $k$ is the first index of each odd-Aztec-rectangle-shaped hole. 
%and the values of $m+k$ and $k$ are readily seen once a concrete region is specified.
However, in order to stress the fact that geometrically the regions look different for $m$ even or odd (for instance, the regions are symmetric in the former case, but only nearly-symmetric in the latter), we include a prime symbol in the notation of the latter.} $AR'_{m+k,m+k+a}\setminus O_{k,a_1-k}^{(c_1)}\cup\cdots\cup O_{k,a_s-k}^{(c_s)}$ (see Figure \ref{fbb} for an example).

The corresponding result is stated below.
%It involves the following modification of the graphs $R(m,n;S)$. Let $m$ be odd, and let $m<n$. Given $T=\{t_1,t_2,\dotsc,t_{m}\}\subset[n]$, let $R'(m,n;T)$ be the graph obtained from the Aztec rectangle graph $A_{m,n}$ by deleting from $\ell'$ all but the $t_1$-st, $t_2$-nd$,\,\dotsc,$\,$t_{m}$-th vertices (from left to right) from $\ell'$. Then by a result due to Krattenthaler \cite[Theorem 9]{Kratt} (see also \cite[Theorem 7.1]{nesy} for a short proof) we have

%%
%\begin{align}
%&\M(R'(m,n;T))=\frac{2^{(m^2+4m-1)/4}}{\prod_{i=1}^{(m-1)/2}(i-1)!\prod_{i=1}^{%(m+1)/2}(i-1)!}
%\nonumber
%\\[10pt]  
%&\ \ \ \ \ \ \ \ \ \ \ \ \ \ \ \ \ \ \ \ \ \ \ \  \ \ \ \ \ \ \ \  
%\times
%\prod_{1\leq i<j\leq(m-1)/2}(t_{2j}-t_{2i})
%\prod_{1\leq i<j\leq(m+1)/2}(t_{2j-1}-t_{2i-1}).
%\label{ebc}
%\end{align}
%%

\begin{theo}
\label{tbb}
Let $m$ be odd, and let $k$ be a positive integer with $k\leq\underset{i}{\min a_i}$. Then we have
\begin{align}
&
\M\left(AR'_{m+k,m+k+a}\setminus O_{k,a_1-k}^{(c_1)}\cup\cdots\cup O_{k,a_s-k}^{(c_s)}\right)
\nonumber
\\
&\ \ \ \ \ \ \ \ \ \ \ \ \ \ \ \ \ \ \ \ \ \ \ \ \ \ \ \ \ \ \ \ \ \ \ \
=
2^{{k+1\choose2}(s+1)+km}
\M\left(AR'_{m,m+a}\setminus O_{0,a_1}^{(c_1)}\cup\cdots\cup O_{0,a_s}^{(c_s)}\right)
\nonumber
\\
&\ \ \ \ \ \ \ \ \ \ \ \ \ \ \ \ \ \ \ \ \ \ \ \ \ \ \ \ \ \ \ \ \ \ \ \
=2^{{k+1\choose2}(s+1)+km}
\M\left(R'(m,m+a;A_1^{(c_1)}\cup\cdots\cup A_s^{(c_s)})\right),
\label{ebd}
\end{align}
where $A_i^{(c_i)}$ is the set of $a_i+1$ consecutive integers centered at $c_i$, and the graph $R'(m,n;T)$ is defined in Section $2$.
\end{theo}

%TO DO: Verify the exponent of 2.

%It seems it is correct (see p.83, 2nd square paper notes).

\begin{figure}[t]
\centerline{
\hfill
{\includegraphics[width=0.35\textwidth]{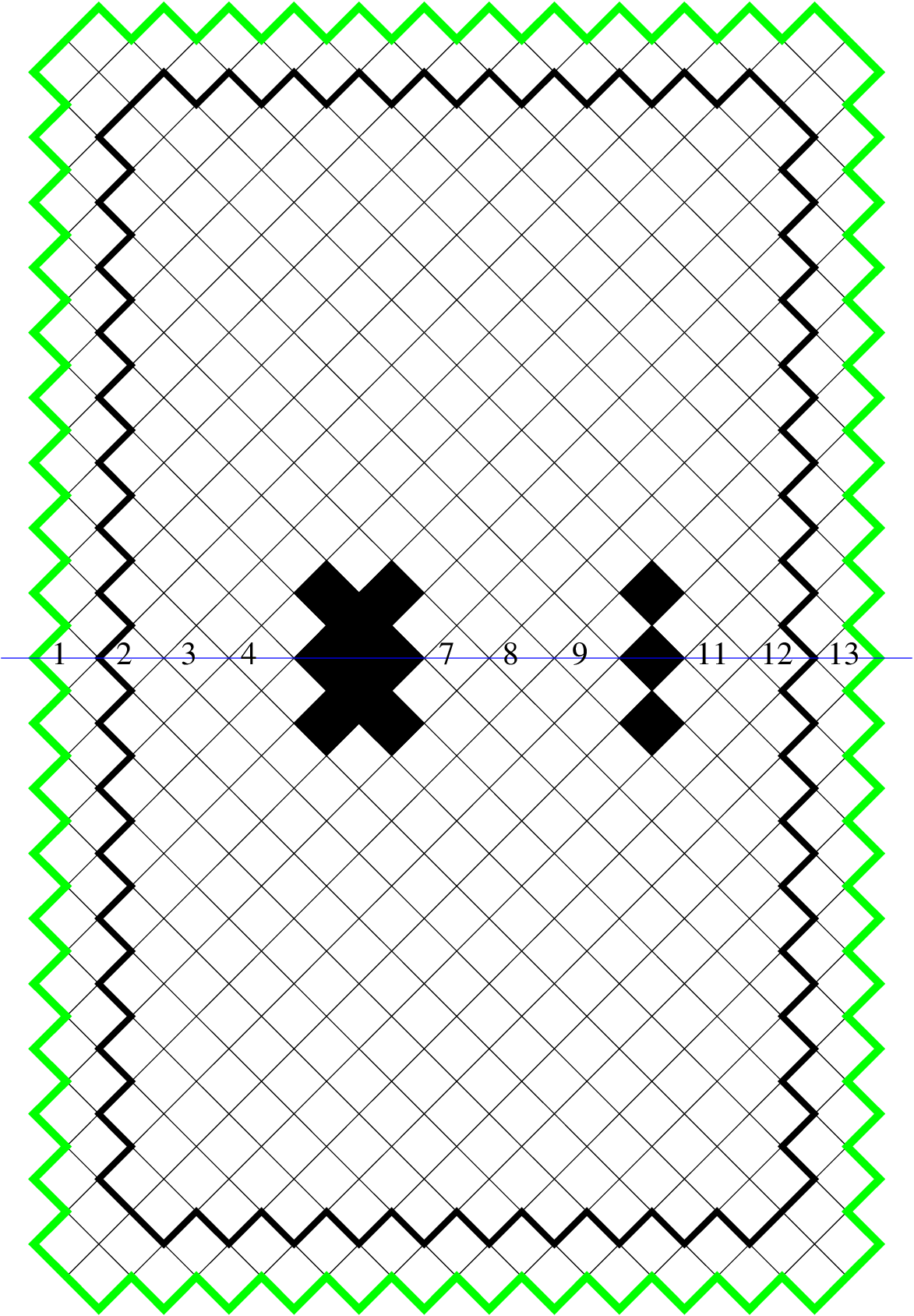}}
\hfill
{\includegraphics[width=0.325\textwidth]{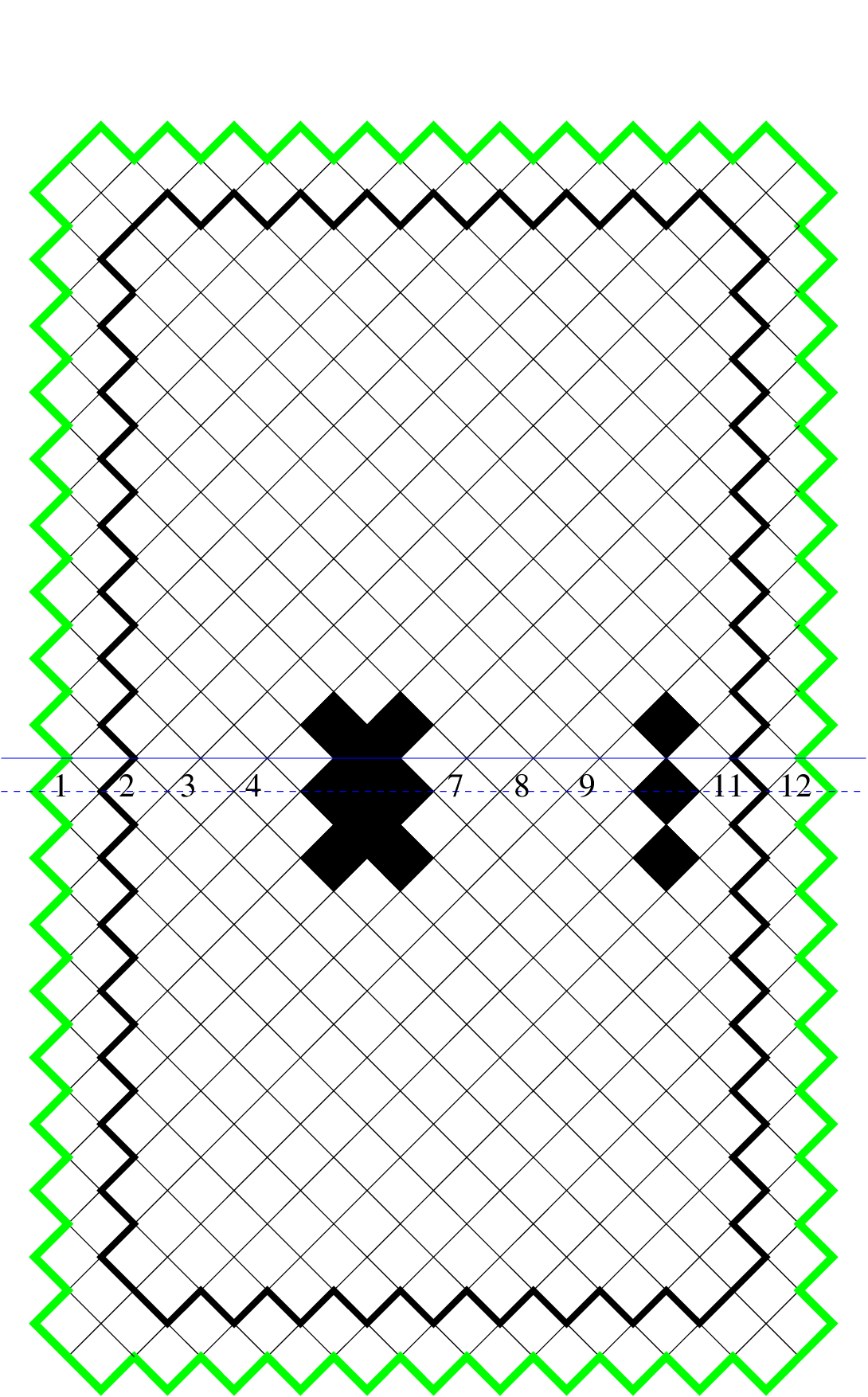}}
\hfill
}
\vskip0.1in
\caption{The case {\it height} $>$ {\it width}. {\it Left.} An instance of an Aztec rectangle region with odd windows and odd frame height: the frame (shown in green) is $AR_{m,m-a}$ with $m=19$, $a=7$, and the region is $AR''_{m-k,m-a-k}\setminus O_{k,a_1-k}^{(c_1)}\cup\cdots\cup O_{k,a_s-k}^{(c_s)}$, where the number of windows is $s=2$, their shapes $O_{k,a_1-k},\dotsc,O_{k,a_s-k}$ are $O_{2,1}$ and $O_{2,0}$ (so $k=2$, $a_1=3$, $a_2=2$), and they are placed so that their centers $c_1,\dotsc,c_s$ are $c_1=5\frac12$ and $c_2=10$. {\it Right.} An Aztec rectangle region with odd windows and even frame height: the frame (shown in green) is $AR_{m,m-a}$ with $m=18$, $a=7$, and the region is $AR'''_{m-k,m-a-k}\setminus O_{k,a_1-k}^{(c_1)}\cup\cdots\cup O_{k,a_s-k}^{(c_s)}$, all parameters except $m$ having the same values as for the region on the left (by definition, $a=\sum_{i=1}^s (a_i+1))$.} 
\vskip-0.1in
\label{fbc}
\end{figure}

The remaining two families of regions are obtained starting with Aztec rectangles in which the height is bigger than the width (see Figure \ref{fbc}). Start with the Aztec rectangle $AR_{m,m-a}$ --- this will be now the frame --- and center {\it inside} it an Aztec rectangle region $AR_{m-k,m-k-a}$. Suppose $m$ is odd. Then it is not hard to verify that we can produce a balanced region from $AR_{m-k,m-k-a}$ by making $s$ windows in the shape of the odd Aztec rectangles $O_{k,a_1-k},\dotsc,O_{k,a_s-k}$ symmetrically along the horizontal symmetry axis $\ell$. Label the unit squares along $\ell$ from left to right $1,2,\dotsc$, starting with the leftmost square that is inside the frame. We denote by $AR''_{m-k,m-a-k}\setminus O_{k,a_1-k}^{(c_1)}\cup\cdots\cup O_{k,a_s-k}^{(c_s)}$ the region obtained by placing the windows so that the center of $O_{k,a_i-k}$ has coordinate $c_i$ on $\ell$. The number of its domino tilings is given by the following formula and equation \eqref{ebaz}.

\begin{theo}
\label{tbc}
For odd $m$ and $0<k\leq\min a_i$ we have
\begin{align}
&
\M\left(AR''_{m-k,m-k-a}\setminus O_{k,a_1-k}^{(c_1)}\cup\cdots\cup O_{k,a_s-k}^{(c_s)}\right)
\nonumber
\\
&\ \ \ \ \ \ \ \ \ \ \ \ \ \ \ \ \ \ \ \ \ \ \ \ \ \ \ \ \ \ \ \ \ \ \ \
=
2^{{k+1\choose2}(s+1)-k(m+1)}
\M\left(AR''_{m,m-a}\setminus O_{0,a_1}^{(c_1)}\cup\cdots\cup O_{0,a_s}^{(c_s)}\right)
\nonumber
\\
&\ \ \ \ \ \ \ \ \ \ \ \ \ \ \ \ \ \ \ \ \ \ \ \ \ \ \ \ \ \ \ \ \ \ \ \
=2^{{k+1\choose2}(s+1)-k(m+1)}
\M\left(R''(m,m-a;A_1^{(c_1)}\cup\cdots\cup A_s^{(c_s)})\right),
\label{ebe}
\end{align}
where $A_i^{(c_i)}$ is the set of $a_i+1$ consecutive integers centered at $c_i$, and the graph $R''(m,n;S)$ is defined in Section $2$.
\end{theo}

The definition in the last case --- when the height is bigger than the width, and the frame height is even --- is very similar to the previous one. The only difference is that now, in order to balance the Aztec rectangle region $AR_{m-k,m-a-k}$, we need to place the odd Aztec windows symmetrically with respect to $\ell'$, so that the resulting region is only nearly-symmetric; we denote it by $AR'''_{m-k,m-k-a}\setminus O_{k,a_1-k}^{(c_1)}\cup\cdots\cup O_{k,a_s-k}^{(c_s)}$. The number of its domino tilings are given by the following result.

\begin{theo}
\label{tbd}
For even $m$ and $0<k\leq\min a_i$ we have
\begin{align}
&
\M\left(AR'''_{m-k,m-k-a}\setminus O_{k,a_1-k}^{(c_1)}\cup\cdots\cup O_{k,a_s-k}^{(c_s)}\right)
\nonumber
\\
&\ \ \ \ \ \ \ \ \ \ \ \ \ \ \ \ \ \ \ \ \ \ \ \ \ \ \ \ \ \ \ \ \ \ \ \
=
2^{{k+1\choose2}(s+1)-k(m+1)}
\M\left(AR'''_{m,m-a}\setminus O_{0,a_1}^{(c_1)}\cup\cdots\cup O_{0,a_s}^{(c_s)}\right)
\nonumber
\\
&\ \ \ \ \ \ \ \ \ \ \ \ \ \ \ \ \ \ \ \ \ \ \ \ \ \ \ \ \ \ \ \ \ \ \ \
=2^{{k+1\choose2}(s+1)-k(m+1)}
\M\left(R'''(m,m-a;A_1^{(c_1)}\cup\cdots\cup A_s^{(c_s)})\right),
\label{ebf}
\end{align}
where $A_i^{(c_i)}$ is the set of $a_i+1$ consecutive integers centered at $c_i$, and the graph $R'''(m,n;S)$ is defined in Section $2$.
\end{theo}

%Note that exponent of 2 in Theorems \ref{tba} and \ref{tbb} can also be written as ${k+1\choose2}(s+1)+km$, which makes it very similar to the exponents in Theorems \ref{tbc} and \ref{tbd} above.
%--- but it would be even more similar if above we had $-km$ instead of $-k(m+1)$ at the exponent.

\parindent0pt
%\medskip
%{\it Remark $1$.} Note that the common value of the exponents of 2 in Theorems \ref{tba} and \ref{tbb} can also be written as ${k+1\choose2}(s+1)+km$, which makes it quite similar to the exponents in Theorems~\ref{tbc} and \ref{tbd}.

\medskip
{\it Remark $1$.} The special case of Theorems \ref{tba}--\ref{tbd} when $a_1=\cdots=a_s=k$ provide simple product formulas for the number of perfect matchings of the graphs obtained from Aztec rectangles by making same-length vertical slits so that the resulting graphs are horizontally symmetric or nearly-symmetric.

\parindent15pt
In particular, for $s=1$, this gives a pleasing counterpart of the Aztec rectangle with two intrusions (see Figure 9 and formula (3.6) in \cite{df}), which in turn is a special case of Krattenthaler's results \cite[Theorems 13--14]{Kratt} (more precisely, in \cite[Figure 9]{df} two consecutive runs of vertices that communicate with the exterior are removed from along a grid diagonal from an Aztec rectangle graph, while in the current situation one consecutive run --- positioned symmetrically or nearly-symmetrically --- is removed from a grid diagonal).

\medskip
\parindent0pt
{\it Remark $2$.} Our arguments that prove Theorems \ref{tba}--\ref{tbd} also work when the odd Aztec windows are lined up symmetrically along {\it any} horizontal grid diagonal, not just for $\ell$ or $\ell'$. Namely, if this grid diagonal is $d$ units below\footnote{ One unit here means translating $\ell$ one unit southeast.} $\ell$, our arguments imply that in this more general situation the number of domino tilings is equal to the same power of two as in Theorems \ref{tba}--\ref{tbd}, times the number of perfect matchings of graphs analogous to the ones in Figure \ref{fbaa}, but with the unit holes lined up along horizontal lines that are $d$ units below $\ell$. Formulas expressing the number of perfect matchings of the latter as $\lfloor d/2 \rfloor$-fold sums of round quantities were given by Krattenthaler in \cite[Theorems 11--12]{Kratt}; see also \cite[Theorem 7.3]{nesy} for a simpler expression (and a different proof) in the case $d=2$.

\medskip
\parindent0pt
{\it Remark $3$.} %Say how electrostatic conjecture (up to multiplicative constant; i.e.,  weak SSP) follows (if we also assume that $\tilde\omega=const\cdot\bar\omega$); see pp. 109-110, 2nd sq. paper notes.  
In our earlier work \cite{gd}, we conjectured that the asymptotic behavior of the bulk correlation of holes in domino tiling systems is given by the following
simple formula.
Given the holes $O_1,\dotsc,O_n$ on the square lattice, we conjectured that their correlation in the bulk $\omega(O_1,\dotsc,O_n)$ has asymptotics given by
\begin{equation}
\omega(O_1,\dotsc,O_n)\sim c \prod_{1\leq i<j\leq n} \de(O_i,O_j)^{\frac12 \q(O_i)\q(O_j)}
\label{EC}
\end{equation}
in the limit of large mutual separations between the holes,
where $\de$ is the Euclidean distance, $\q(O)$ denotes the charge of the hole $O$ (the number of white unit squares minus the number of black unit squares in $O$), and the constant $c$ depends only on the multiset of shapes of the holes $O_1,\dotsc,O_n$.
We call equation \eqref{EC} the {\it electrostatic conjecture}\footnote{
  To be precise, the electrostatic conjecture \cite[Equation (2.12)]{gd} (as well as its lozenge tilings version \cite[Conjecture 1]{ov}) specifies also the multiplicative constant, as being the product of the individual correlations of the holes. These explicit forms hold only for special, carefully defined correlations. 
  For the purposes of this discussion we will stick to the original version we presented in \cite[Section 14]{sc}, which does not specify the multiplicative constant, and holds for a larger class of naturally defined correlations --- the correlations $\tilde\omega$ and $\tilde\omega'$ presented here are two examples, each with specific advantages for calculations.
 % it is stated there for the specific correlation $\bar\omega$ defined by \cite[Equations (2.2)--(2.4)]{gd}.
 %There are several other natural correlations of holes that can be defined --- the $\tilde\omega$ and $\tilde\omega'$ presented here are two examples --- which have specific advantages for calculations, and are either conjectured or can be proved (as it is the case here) to differ from $\bar\omega$ only by a multiplicative constant. Then we say that these correlations satisfy the weak form of the electrostatic conjecture \cite[Equation (2.12)]{gd} if they satisfy it up to a multiplicative constant.
}
(for the lozenge tilings version, see \cite[Conjecture 1]{ov} and \cite[Section 14]{sc}; for general doubly periodic lattices, see \cite{ge}).

\parindent12pt
We point out here how the exact formulas given by Theorem \ref{tba} can be used to prove the electrostatic conjecture in the case when the collection of holes consists of odd Aztec rectangles of the same height, placed symmetrically along a horizontal lattice diagonal. For this, we define their correlation $\tilde\omega'$ as follows.

%The exact formulas provided by Theorem \ref{tba} can be used to prove the weak form
%\footnote{ The electrostatic conjecture \cite[Equation (2.12)]{gd} specifies also the multiplicative constant, as being the product of the individual correlations of the holes, and is stated for the specific correlation $\bar\omega$ defined by \cite[Equations (2.2)--(2.4)]{gd}. There are several other natural correlations of holes that can be defined --- the $\tilde\omega$ and $\tilde\omega'$ presented here are two examples --- which have specific advantages for calculations, and are either conjectured or can be proved (as it is the case here) to differ from $\bar\omega$ only by a multiplicative constant. Then we say that these correlations satisfy the weak form of the electrostatic conjecture \cite[Equation (2.12)]{gd} if they satisfy it up to a multiplicative constant.}
%of the electrostatic conjecture \cite[Equation (2.12)]{gd} for the correlation $\tilde\omega'$ of collinear, same-height odd Aztec rectangles defined as follows.

\parindent15pt
First, we recall from \cite[Section 5]{gd} the definition of the correlation $\tilde\omega$ of a collection of consecutive runs of monomers along a diagonal of the square grid:
\begin{equation}
\tilde\omega(O_{0,a_1}^{(c_1)},\dotsc,O_{0,a_s}^{(c_s)})
=
\lim_{m\to\infty}
\frac{\M(AR_{m,m+a}\setminus O_{0,a_1}^{(c_1)}\cup\cdots\cup O_{0,a_s}^{(c_s)})}
     {\M(AR_{m,m+a}\setminus O_{0,a_1}^{(c^0_1)}\cup\cdots\cup O_{0,a_s}^{(c^0_s)})},
\label{ebxa}
\end{equation}
where the locations $c^0_1,\dotsc,c^0_s$ of the centers of the runs at the denominator on the right hand side are reference positions, with $c^0_1=c_1$, and $c^0_2,\dotsc,c^0_s$ chosen so that these $s$ runs of monomers are contiguous (i.e., form together one single consecutive run).   

The correlation $\tilde\omega'$ is patterned on the above definition:
%(and uses the same values of the reference coordinates $c^0_1,\dotsc,c^0_s$ as above):
given holes in the shape of odd Aztec rectangles $O_{k,a_1-k}^{(c_1)},\dotsc,O_{k,a_s-k}^{(c_s)}$ along a grid diagonal (which we think of as being the symmetry axis $\ell$ from Theorem \ref{tba}), we define
\begin{equation}
\tilde\omega'(O_{k,a_1-k}^{(c_1)},\dotsc,O_{k,a_s-k}^{(c_s)})
=
\lim_{m\to\infty}
\frac{\M(AR_{m+k,m+k+a}\setminus O_{k,a_1-k}^{(c_1)}\cup\cdots\cup O_{k,a_s-k}^{(c_s)})}
   {\M(AR_{m+k,m+k+a}\setminus O_{k,a_1-k}^{(c^0_1)}\cup\cdots\cup O_{k,a_s-k}^{(c^0_s)})},
\label{ebxb}
\end{equation}
where the values of the reference coordinates $c^0_1,\dotsc,c^0_s$ which specify the centers of the holes are the same as in \eqref{ebxa} (which is the special case $k=0$ of \eqref{ebxb}; in particular, for $k>0$ the Aztec rectangle holes in \eqref{ebxb} do not touch each other).

Two applications of Theorem \ref{tba} give
\begin{equation}
\M(AR_{m+k,m+k+a}\setminus O_{k,a_1-k}^{(c_1)}\cup\cdots\cup O_{k,a_s-k}^{(c_s)})
=
2^{{k+1\choose2}(s+1)+km}
\M(AR_{m,m+a}\setminus O_{0,a_1}^{(c_1)}\cup\cdots\cup O_{k,a_s-k}^{(c_s)})
\end{equation}
and
\begin{equation}
\M(AR_{m+k,m+k+a}\setminus O_{k,a_1-k}^{(c^0_1)}\cup\cdots\cup O_{k,a_s-k}^{(c^0_s)})
=
2^{{k+1\choose2}(s+1)+km}
\M(AR_{m,m+a}\setminus O_{0,a_1}^{(c^0_1)}\cup\cdots\cup O_{k,a_s-k}^{(c^0_s)}).
\end{equation}
Taking their ratio and letting $m\to\infty$ we obtain
\begin{equation}
\tilde\omega'(O_{k,a_1-k}^{(c_1)},\dotsc,O_{k,a_s-k}^{(c_s)})
=
\tilde\omega(O_{0,a_1}^{(c_1)},\dotsc,O_{0,a_s}^{(c_s)}).
\label{ebxc}
\end{equation}
The arguments in the proof of \cite[Lemma 5.1]{gd} imply that, when restricted to holes that are consecutive runs of monomers lined up along a common grid diagonal, the correlation $\tilde\omega$ defined above and the correlation $\bar\omega$ defined by \cite[Equations (2.2)--(2.4)]{gd} differ only by a multiplicative constant. Therefore, Theorem 3.1 of \cite{gd}, which proves the strong version of the electrostatic conjecture for the correlation $\bar\omega$ of runs of monomers along a common diagonal, implies that the correlation $\tilde\omega$ of such defects satisfies the weak form of the electrostatic conjecture. By \eqref{ebxc}, this gives in turn that the correlation $\tilde\omega'$ of odd Aztec rectangles of the same height lined up along a diagonal also satisfies the weak form of the electrostatic conjecture.

\medskip
\parindent0pt
{\it Remark $4$.} Note that by our definitions, for the last two types of Aztec rectangles with odd windows it is allowed for the leftmost or rightmost window (or both) to communicate with the outside --- this happens precisely when those extremal windows are flush with the outer boundary of the region\footnote{ This is not possible for the first two types, as for them the frame (which contains the windows) is strictly inside their outer boundary.}. It turns out that in these situations some new families of regions arise, which are still covered by the proofs of Theorems \ref{tbc} and \ref{tbd}.

\begin{figure}[t]
\centerline{
\hfill
{\includegraphics[width=0.40\textwidth]{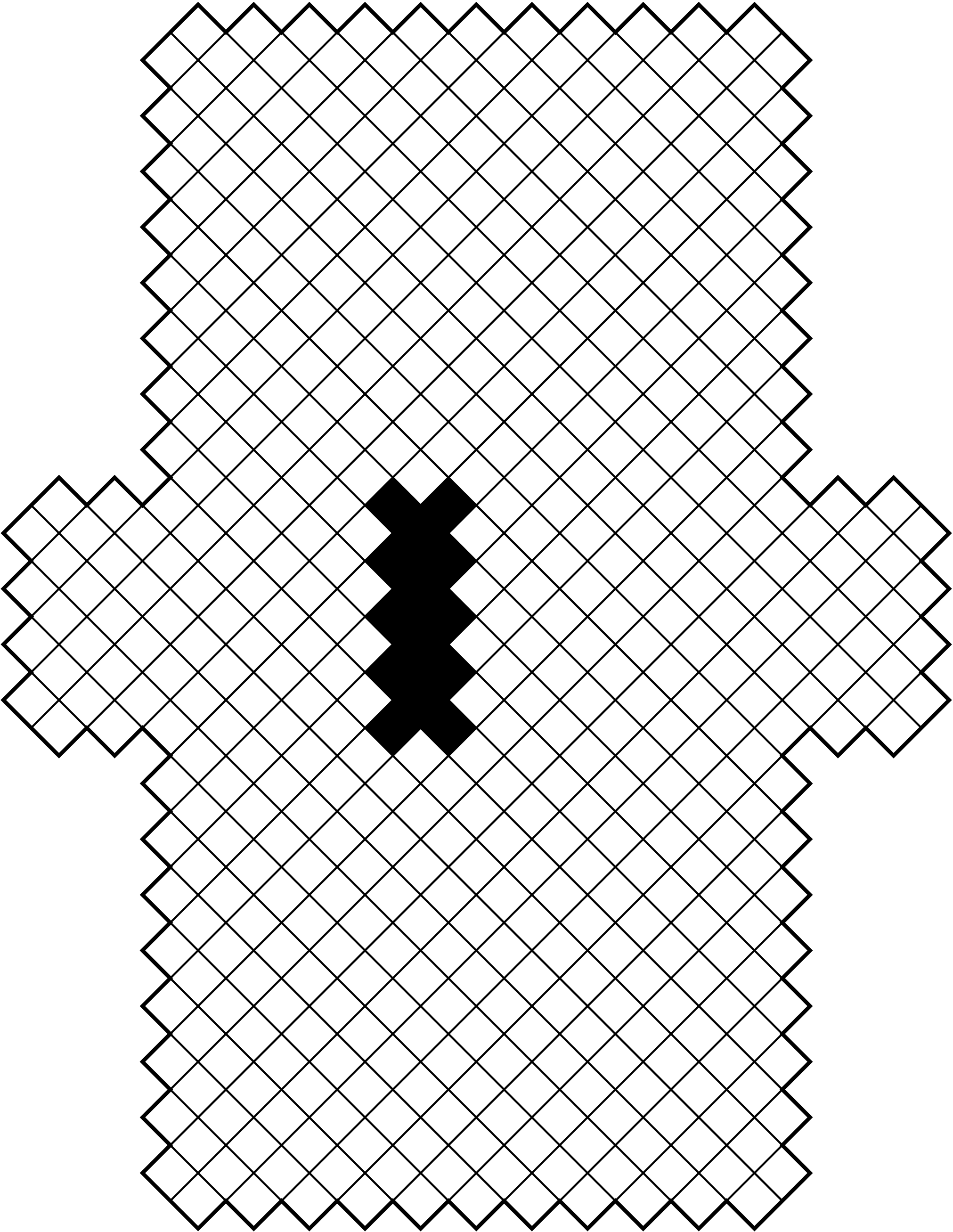}}
\hfill
}
\vskip0.1in
\caption{ The cruciform region $C_{m-a-k,k}^{(m-1)/2-k,a_s+1-k,(m-1)/2-k,a_1+1-k}$ with one hole of shape $O_{k,a_2-k}$, for $m=25$, $k=4$, $s=3$, $a_1=1$, $a_2=5$, $a_3=1$ (so $a=\sum_{i=1}^s (a_i+1)=10$).}
\vskip-0.1in
\label{fbd}
\end{figure}

\parindent15pt
In order to arrive at them, it will be helpful to think of the regions involved in Theorem \ref{tbc} as being the planar duals of graphs obtained by repeated applications of the complementation theorem \cite[Theorem 2.1]{CT} to the graph $R''(m,n;T)$ on the right hand side of \eqref{ebe} (see Figure~\ref{fdx} and its caption).

Suppose for instance that both the first and the last run of consecutive missing monomers in $R''(m,n;T)$ communicate with the outside (the picture on the left in Figure~\ref{fdx} illustrates such a case). Then, after $k$ successive applications of the above mentioned complementation theorem, it is not hard to see that the outside boundary of $R''(m,n;T)$ becomes the planar dual of a certain cruciform region --- these regions were introduced in our earlier work \cite{df} --- namely the cruciform region $C_{m-a-k,k}^{(m-1)/2-k,a_s+1-k,(m-1)/2-k,a_1+1-k}$. As long as $k\leq a_1$ and $k\leq a_s$, these are $AR''$-type Aztec rectangles with odd windows. But for larger values of $k$ --- in particular, if $k$ is greater or equal than both $a_1+1$ and $a_s+1$ --- then the obtained regions do have actual cruciform shape; Figure \ref{fbd} shows such an example (it is the region whose planar dual is the graph on the right in Figure \ref{fdx}).

%The effect of these $k$ successive applications on the holes is the same as in Theorems \ref{tba}--\ref{tbd}, namely that a hole consisting of $a_i+1$ consecutive monomers on $\ell$ becomes the odd Aztec rectangle $O_{k,a_i-k}$.
It is not hard to see (and it is part of our proof of Theorem \ref{tbc}) that the effect on the exponent of 2 coming from the repeated applications of the complementation theorem contributed by the first (and similarly the last) hole is the same whether or not that hole communicates with the exterior, as long as $k\leq a_1$. It turns out (as we show in the proof of Theorem \ref{tbc}) that the same is true for $k\geq a_1+1$.
%(??--Check this!). We checked!
We obtain that the right hand side of \eqref{ebe} also gives the number of domino tilings of the cruciform region $C_{m-a-k,k}^{(m-1)/2-k,a_s+1-k,(m-1)/2-k,a_1+1-k}$ with holes $O_{k,a_2-k},\dotsc,O_{k,a_{s-1}-k}$ placed symmetrically along the horizontal symmetry axis; an example is shown in Figure \ref{fbd}. Our proof of Theorem~\ref{tbd} also covers an analogous family of cruciform regions with holes.

%\medskip
%TO DO: Check exponent of 2 above --- in fact, check all 4 theorems above for the regions in Figures 3-5.

%Double check the exponent above.

%\medskip
%    {\it Remark 1.} In the above results we assumed that the width of the Aztec rectangles with windows is greater than their height. Counterparts for the case when the relative order of these dimensions is reversed can easily be written down (in that case, the resulting regions are symmetric for $m$ odd, and the windows are just off the symmetry axis for $m$ even). The corresponding statements are proved by the very same arguments that prove the above results (the explicit product formulas follow using Theorems~8 and 10 of \cite{Kratt}). In order to keep the focus, we do not include them here.

\section{Toroidal Aztec rectangles with odd Aztec windows}

Denote by $T_{m,n}$ the graph obtained from the Aztec rectangle $AR_{m,n}$ by identifying corresponding vertices along its top and bottom, and along its left and right boundaries (Figure \ref{fca} shows $T_{13,17}$ with some holes in it).

Since it is obtained from $AR_{m,n}$ by identifying vertices of the same color, the graph $T_{m,n}$ is bipartite. Fix a black and white proper coloring of its vertices. We will make holes in this graph by removing sets of vertices that induce subgraphs isomorphic to odd Aztec rectangles. Each such collection of vertices can be placed in two different ways: if, under its placement, the majority of the vertices are white, we say that the odd Aztec rectangle hole is {\it white-placed}, while if the majority are black, we say that the hole is {\it black-placed}$\,$\footnote{ The difference between the number of majority and minority vertices in $O_{k,l}$ is $k+l+1$, so there is always a strict majority.}. 

One convenient property of the toroidal Aztec rectangle $T_{m,n}$ is that, for all values of $m$ and~$n$, the number of its white and black vertices is the same (indeed, pairing each white vertex with its northeastern neighbor\footnote{ Under the toroidal boundary conditions we described. } gives a bijection between the two color classes). 
%(as mentioned in Section 3, if $m\leq n$ and its topmost vertices are white, the Aztec rectangle $AR_{m,n}$ has $n-m$ more white than black vertices; identifying the vertices on top with those on the bottom reduces the number of white vertices by $m$, while the identification of the vertices along the left and right boundary reduces the number of black vertices by $n$).
Therefore, in order to end up with a balanced graph, the number of white and black vertices in the union of the holes has to be the same. In other words, if we define the {\it charge} $\q(O)$ of a hole $O$ to be the number of white vertices in $O$ minus the number of black vertices in $O$, what we need is that the total charge (i.e., the sum of the individual charges) of the holes is zero.

\begin{figure}[t]
\centerline{
\hfill
{\includegraphics[width=0.50\textwidth]{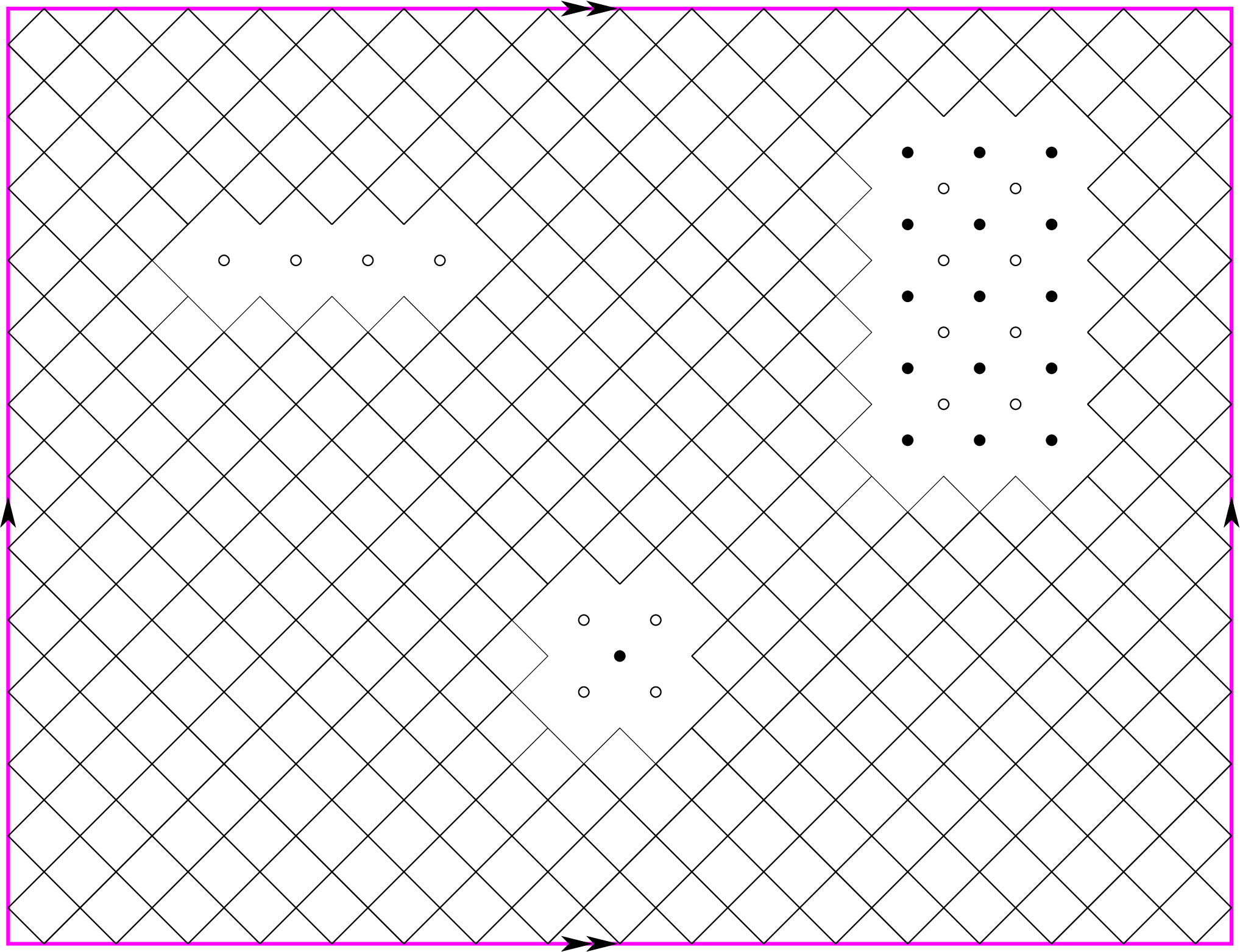}}
\hfill
{\includegraphics[width=0.35\textwidth]{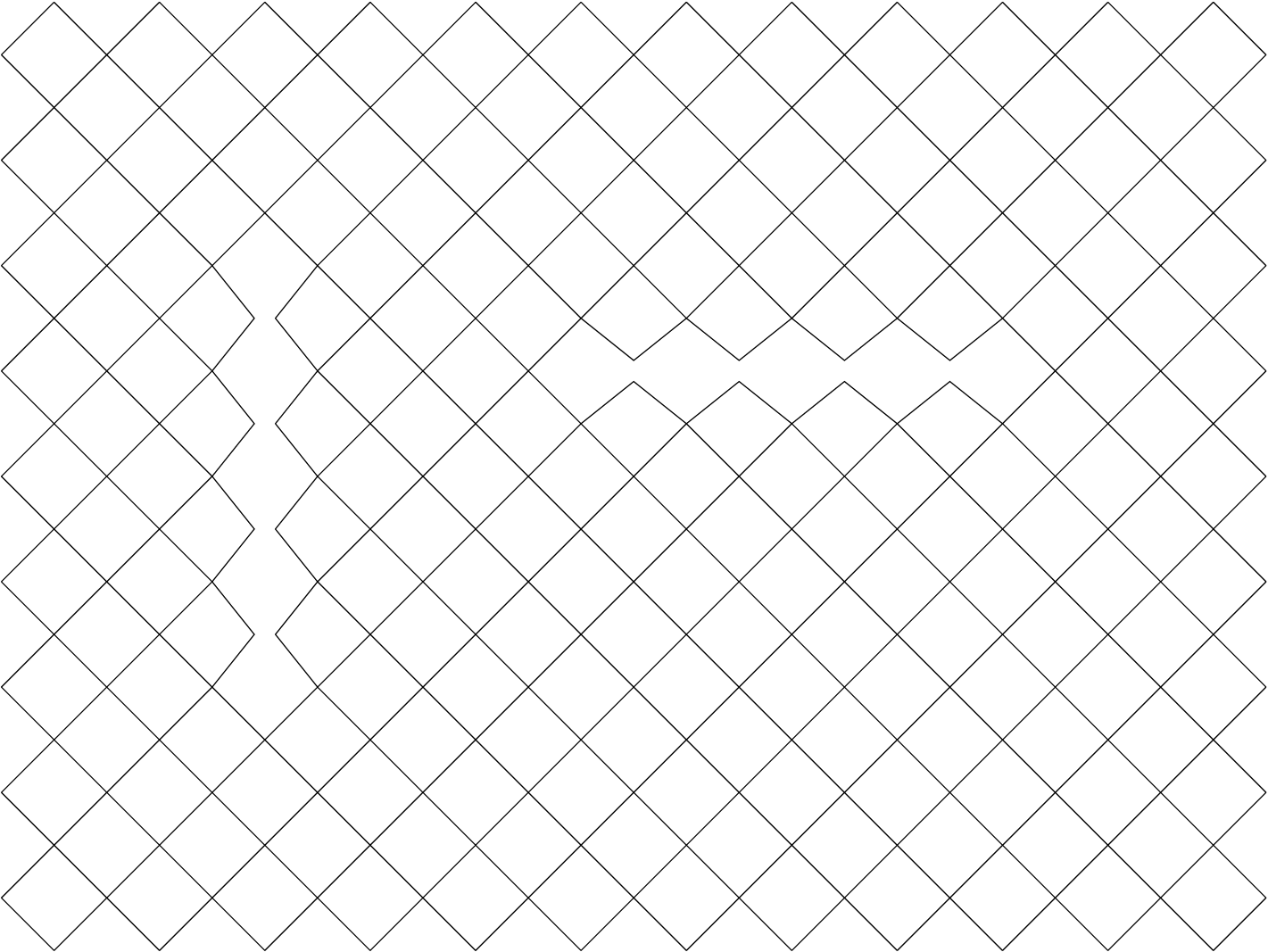}}
\hfill
}
\vskip0.1in
\caption{{\it Left.} A toroidal Aztec rectangle graph with three odd Aztec windows: two white-placed (in the shape of $O_{0,3}$ and $O_{1,1}$) and one black-placed (of shape $O_{4,2}$). {\it Right.} A vertical and a horizontal defect consisting of four contiguous separations.} 
\vskip-0.1in
\label{fca}
\end{figure}

One readily checks that the charge of a white-placed odd Aztec rectangle $O_{k,l}$ is $k+l+1$, while the charge of a black-placed odd Aztec rectangle $O_{k,l}$ is $-(k+l+1)$.

%Let $a_1,\dotsc,a_s,b_1,\dotsc,b_s,c_1,\dotsc,c_t,d_1,\dotsc,d_t\geq0$ be integers.
%Denote by $T_{m,n}(O_{a_1,b_1},\dotsc,O_{a_s,b_s}; O_{c_1,d_1},\dotsc,O_{c_t,d_t})$ the graph obtained from the toroidal Aztec rectangle $T_{m,n}$ by making in it white-placed holes $O_{a_1,b_1},\dotsc,O_{a_s,b_s}$ and black-placed holes $O_{c_1,d_1},\dotsc,O_{c_t,d_t}$.

Let $k,l\geq0$ be integers. Given a hole $O$ of shape $O_{k,l}$ with $l\geq1$ which is white-placed, we define the {\it evolved form of $O$} to be the hole $\e(O)$ of shape $O_{k+1,l-1}$ which has the same center as~$O$. For a a hole $O$ of shape $O_{k,l}$ with $k\geq1$ which is {\it black}-placed, we define its evolved form to be the hole $\e(O)$ of shape $O_{k-1,l+1}$, again placed so that it has the same center as~$O$. If $O$ is white-placed of shape $O_{k,0}$, we define $\e(O)$ to be the defect consisting of $k+1$ contiguous vertical separations, placed so that its center agrees with the center of $O$. For $O$ a black-placed hole of shape $O_{0,l}$, define $\e(O)$ to be the defect consisting of $l+1$ contiguous {\it horizontal} separations, placed so that its center agrees with the center of $O$. The picture on the right in Figure \ref{fca} shows examples of such defects.

The following fact will be useful for the phrasing of the results in this section.

\begin{figure}[t]
\centerline{
\hfill
{\includegraphics[width=0.25\textwidth]{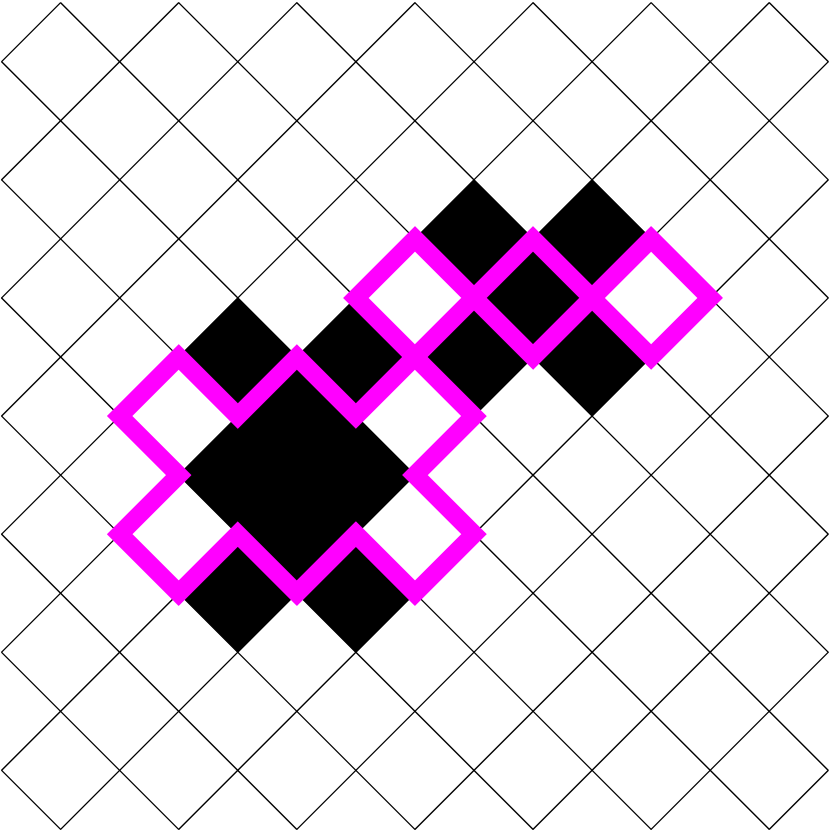}}
\hfill
{\includegraphics[width=0.25\textwidth]{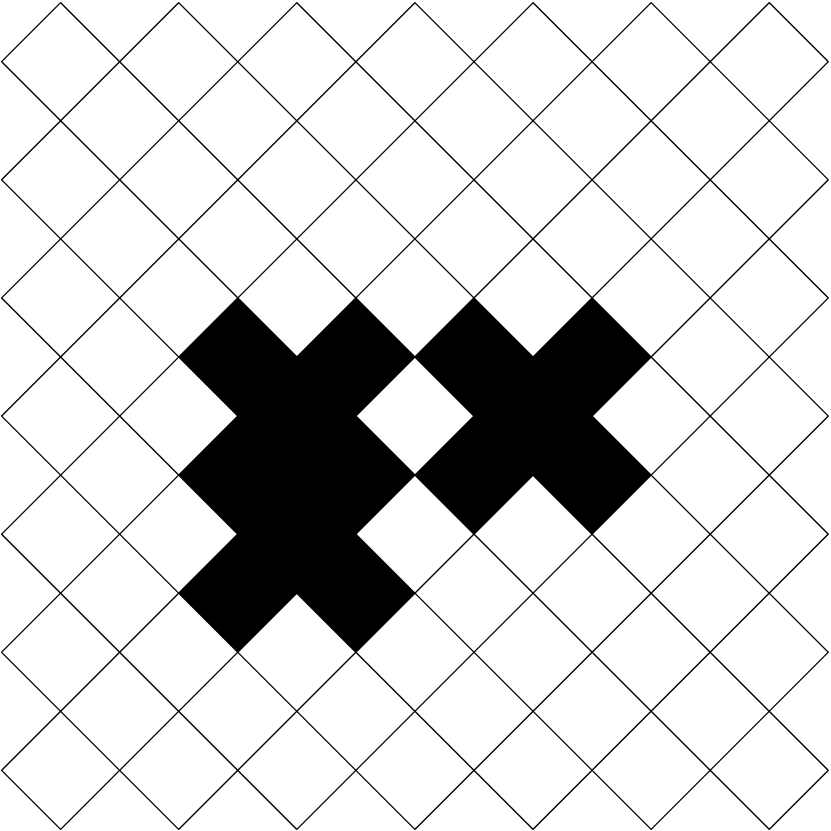}}
\hfill
}
\vskip0.1in
\caption{ {\it Left.} Two holes of odd Aztec rectangle shape (shown in black) that touch vertically with one common corner have disjoint evolved forms (their contours are shown in magenta). {\it Right.} If the holes touch vertically with at least two common corners, a region with these holes has no domino tiling} 
\vskip-0.1in
\label{fcaa}
\end{figure}

\begin{lem}
\label{tcaa}
Let $T_{m,n}$ be a toroidal Aztec rectangle, and let $O_1,\dotsc,O_r$ be holes in it, each in the shape of some odd Aztec rectangle. Assume that the $O_i$'s are mutually disjoint, and that the graph $T_{m,n}\setminus O_1\cup\cdots\cup O_r$ obtained from $T_{m,n}$ by making in it these holes has at least one perfect matching. Then the evolved forms $\e(O_1),\dotsc,\e(O_r)$ of the holes are also mutually disjoint.

\end{lem}

\begin{proof}
  Consider two of the holes, without loss of generality $O_1$ and $O_2$; by assumption, they are disjoint. We claim that if one of them is white-placed and the other black-placed, it follows that $\e(O_1)$ and $\e(O_2)$ are also disjoint, even without using the assumption that $T_{m,n}\setminus O_1\cup\cdots\cup O_r$ has at least one perfect matching. Indeed, it follows from the definition of the evolved form that $O_1$ and $O_2$ evolve in opposite directions. Namely, since one of them is white-placed and the other black-placed, for one of them the evolved form makes the odd Aztec rectangle shape of the hole be one unit wider and one unit less tall, while for the other it is one unit taller and one unit less wide. Due to these movements being compatible and the fact that the evolved form of a hole has the same center as the original hole, the disjoint holes $O_1$ and $O_2$ produce disjoint evolved forms $\e(O_1)$ and $\e(O_2)$.

Suppose now that both $O_1$ and $O_2$ are white-placed holes in the shape of odd Aztec rectangles (the case when they are black-placed is handled similarly). By definition, $\e(O_1)$ and $\e(O_2)$ have odd Aztec rectangle shapes that have the same centers as $O_1$ and $O_2$, and are one unit wider and one unit less tall. Suppose towards a contradiction that $\e(O_1)$ and $\e(O_2)$ are not disjoint. Since $O_1$ and $O_2$ are, this could only happen if $O_1$ and $O_2$, when regarded as holes in a region tiled by dominos, touch in at least two points along a vertical diagonal (see Figure \ref{fcaa}). However, then there is at least one unit square ``island'' between $O_1$ and $O_2$, which cannot be covered by a domino, contradicting the fact that $T_{m,n}\setminus O_1\cup\cdots\cup O_r$ has at least one perfect matching. This completes the proof.
\end{proof}  

The results in this section concern toroidal Aztec rectangle graphs $T_{m,n}$ in which holes $O_1,\dotsc,O_r$ were made, each hole having the shape of an odd Aztec rectangle. We will assume from here on that the removed odd Aztec rectangles are disjoint, and that the total charge of the holes is zero. We call the holes {\it windows}, and the resulting graphs $T_{m,n}\setminus O_1\cup\cdots\cup O_r$ {\it toroidal Aztec rectangles with odd Aztec windows} (or simply toroidal Aztec rectangles with windows).

The {\it flank charge} of a hole $O$ of shape $O_{k,l}$ is defined to be $\f(O)=l$ if the hole $O_{k,l}$ is white-placed, and $\f(O)=-(l+1)$ if the hole $O_{k,l}$ is black-placed.

The way the number of perfect matchings of a toroidal Aztec diamond with odd Aztec windows changes under evolution of windows is given by the following result.

\begin{theo}
\label{tca} For any toroidal Aztec rectangle with odd Aztec windows $O_1,\dotsc,O_r$ which has at least one perfect matching we have
\begin{equation}
\M(T_{m,n}\setminus O_1\cup\cdots\cup O_r)
=
2^{\sum_{i=1}^r \f(O_i)}
\M(T_{m,n}\setminus \e(O_1)\cup\cdots\cup \e(O_r)).
\label{eca}
\end{equation}

\end{theo}

\medskip
\parindent=0pt
{\it Remark $5$.} It seems that the symmetry has been broken when we chose to define the flank charge of an odd Aztec window based on its width. This is resolved by noticing that a definition in terms of heights leads to the same expression as \eqref{eca}. Indeed, the correct definition of flank charge in terms of heights turns out to be $\f(O_{k,l})=-(k+1)$ if $O_{k,l}$ is white-placed, and $\f(O_{k,l})=k$ if $O_{k,l}$ is black-placed. For white-placed holes $O_{k_1,l_1},\dotsc,O_{k_s,l_s}$ and black-placed holes $O_{p_1,r_1},\dotsc,O_{p_t,r_t}$, the exponent of $2$ in \eqref{eca} is $l_1+\cdots+l_s-(r_1+1)-\cdots-(r_t+1)$, while with the above-mentioned alternative definition of flank charge it becomes $-(k_1+1)-\cdots-(k_s+1)+p_1+\cdots+p_t$. The two expressions are equal because the white-placed hole $O_{k_i,l_i}$ has charge $k_i+l_i+1$, the black-placed hole $O_{p_j,r_j}$ has charge $-(p_j+r_j+1)$, and the total charge of the holes is zero.

\parindent=15pt
\medskip
The special case when all white-placed windows are translates of each other, and all black-placed ones are 90 degree rotations of the white-placed ones, is especially simple.

\begin{cor}%[{\bf Congruent windows}]
\label{tcb}
Let $O_1,\dotsc,O_s$ be white-placed odd Aztec windows of shape $O_{k,l}$ and $O_{s+1},\dotsc,O_{2s}$ black-placed odd Aztec windows of shape $O_{l,k}$ in a toroidal Aztec rectangle $T_{m,n}$, and assume that the resulting toroidal Aztec rectangle with windows has at least one perfect matching. For an odd Aztec window $O$ of shape $O_{k,l}$, denote by $O^{+,-}$ $($resp., $O^{-,+}$$)$ the odd Aztec window of shape $O_{k+1,l-1}$ $($resp., $O_{k-1,l+1}$$)$ having the same center as $O$. Then we have

\begin{equation}
\M(T_{m,n}\setminus O_1\cup \cdots \cup O_s\cup O_{s+1}\cup \cdots \cup O_{2s})
=
2^{s(l-k-1)}\M(T_{m,n}\setminus O_1^{+,-}\cup \cdots \cup O_s^{+,-}\cup O_{s+1}^{-,+}\cup \cdots \cup O_{2s}^{-,+}).
\label{ecb}
\end{equation}
\end{cor}

\begin{proof}
Apply Theorem \ref{tca}. By definition, the evolved form of an Aztec window $O$ is $O^{+,-}$ if $O$ is white-placed, and $O^{-,+}$ if $O$ is black-placed. The flank charge of a white-placed $O_{k,l}$ is $l$, while the flank charge of a black-placed $O_{l,k}$ is $-(k+1)$. Equation \eqref{ecb} thus follows directly from \eqref{eca}.
\end{proof}

By repeated application of Corollary \ref{tcb}, we can relate the case when all windows are congruent to~$O_{k,l}$ to that when all are congruent to $O_{0,k+l}$.% --- in other words, diagonal multiplets of size $k+l$.

In the spirit of \cite{ge} --- where a pair of vertices of the same color on a face of a bipartite graph was called a doublet --- we define a {\it diagonal multiplet of length $k$} (also called sometimes a {\it diagonal slit}) to be the union of a consecutive run of $k$ diagonally adjacent vertices. With our drawing of the grid graph, diagonal multiplets are either horizontal or vertical (the leftmost hole in Figure \ref{fca} is a horizontal multiplet of length~4). Alternatively, a horizontal (resp., vertical) $k$-multiplet is just a hole of shape $O_{0,k-1}$ (resp.,~$O_{k-1,0}$).

For an odd Aztec window $O$ of shape $O_{k,l}$, denote by $\h(O)$ (resp., $\ve(O)$) the horizontal (resp., vertical) diagonal multiplet of length $k+l+1$ having the same center as $O$.

\begin{theo}
\label{tcc}
Let $O_1,\dotsc,O_s$ be white-placed odd Aztec windows of shape $O_{k,l}$ and $O_{s+1},\dotsc,O_{2s}$ black-placed odd Aztec windows of shape $O_{l,k}$ in a toroidal Aztec rectangle $T_{m,n}$, and assume that the resulting toroidal Aztec rectangle with windows has at least one perfect matching.
%For an odd Aztec window $O$ of shape $O_{k,l}$, denote by $\h(O)$ $($resp., $\ve(O)$$)$ the horizontal $($resp., vertical$)$ diagonal multiplet of length $k+l+1$ having the same center as $O$.
Then we have

\begin{equation}
%\\[10pt]
\M(T_{m,n}\setminus O_1\cup \cdots \cup O_s\cup O_{s+1}\cup \cdots \cup O_{2s})
=
2^{-slk}\M(T_{m,n}\setminus \h(O_1)\cup \cdots \cup \h(O_s)\cup \ve(O_{s+1})\cup \cdots \cup \ve(O_{2s})).
\label{ecc}
\end{equation}

\end{theo}

\begin{proof}
Apply Corollary \ref{tcb} repeatedly, starting with the graph on the right hand side of \eqref{ecc}, which is
the toroidal Aztec diamond $T_{m,n}$ with $s$ white-placed Aztec windows of shape $O_{0,k+l}$, and $s$ black-placed Aztec windows of shape $O_{k+l,0}$. After $k$ successive applications\footnote{ The reason why it is possible to apply Corollary~\ref{tcb} successively is because the new graph obtained after each application still meets the condition that it has at least one perfect matching. Indeed, this follows from Corollary \ref{tcb} itself: if the number of perfect matchings on the left hand side of \eqref{ecb} is non-zero, then so is the one on the right hand side.} of Corollary~\ref{tcb}, the holes become $s$ white-placed Aztec windows of shape $O_{k,l}$, and $s$ black-placed Aztec windows of shape $O_{l,k}$. In other words, the resulting graph is precisely the one on the left hand side of \eqref{tcc}. The overall contribution of the powers of two coming from equation \eqref{ecb} is
\begin{equation*}
2^{s[(k+l)-0-1)]}\cdot2^{s[(k+l-1)-1-1]}\cdots2^{s[(l+1)-(k-1)-1]}=2^{skl},
\end{equation*}
and the statement follows.
\end{proof}

Specializing the above result to the case $l=0$ we obtain a surprising symmetry.

%In the spirit of \cite{ge}, where a pair of vertices of the same color on a face of a bipartite graph was called a doublet, we define a {\it diagonal multiplet of length $k$} to be the union of a consecutive run of $k$ diagonally adjacent vertices. With our drawing of the grid graph, diagonal multiplets are either horizontal or vertical (the leftmost hole in Figure \ref{fca} is a horizontal multiplet of length~4). Alternatively, a horizontal (resp., vertical) $k$-multiplet is just a hole of shape $O_{0,k-1}$ (resp.,~$O_{k-1,0}$).

\begin{cor}[{\bf Flip symmetry of diagonal multiplets}]
%\begin{cor}[{\bf Flip symmetry of finite size correlation of diagonal multiplets}]
\label{tcd}
Let $O_1,\dotsc,O_{2s}$ be diagonal multiplets of length $k$, $s$ of them horizontal and white-placed, the other $s$ vertical and black-placed, in a toroidal Aztec rectangle $T_{m,n}$, so that the resulting toroidal Aztec rectangle with windows has at least one perfect matching. For a horizontal $($resp., vertical$)$ diagonal multiplet $O$, define its \textnormal{flip} $\overline{O}$ to be the vertical $($resp., horizontal$)$ diagonal multiplet having the same length and the same center as $O$. Then we have

\begin{equation}
\M(T_{m,n}\setminus O_1\cup \cdots \cup O_{2s})
=
\M(T_{m,n}\setminus \overline{O}_1\cup \cdots \cup \overline{O}_{2s}).
\label{ecd}
\end{equation}

\end{cor}

An example illustrating the statement of Corollary \ref{tcd} is shown in Figure \ref{fcb}.

\begin{figure}[t]
\centerline{
\hfill
{\includegraphics[width=0.4\textwidth]{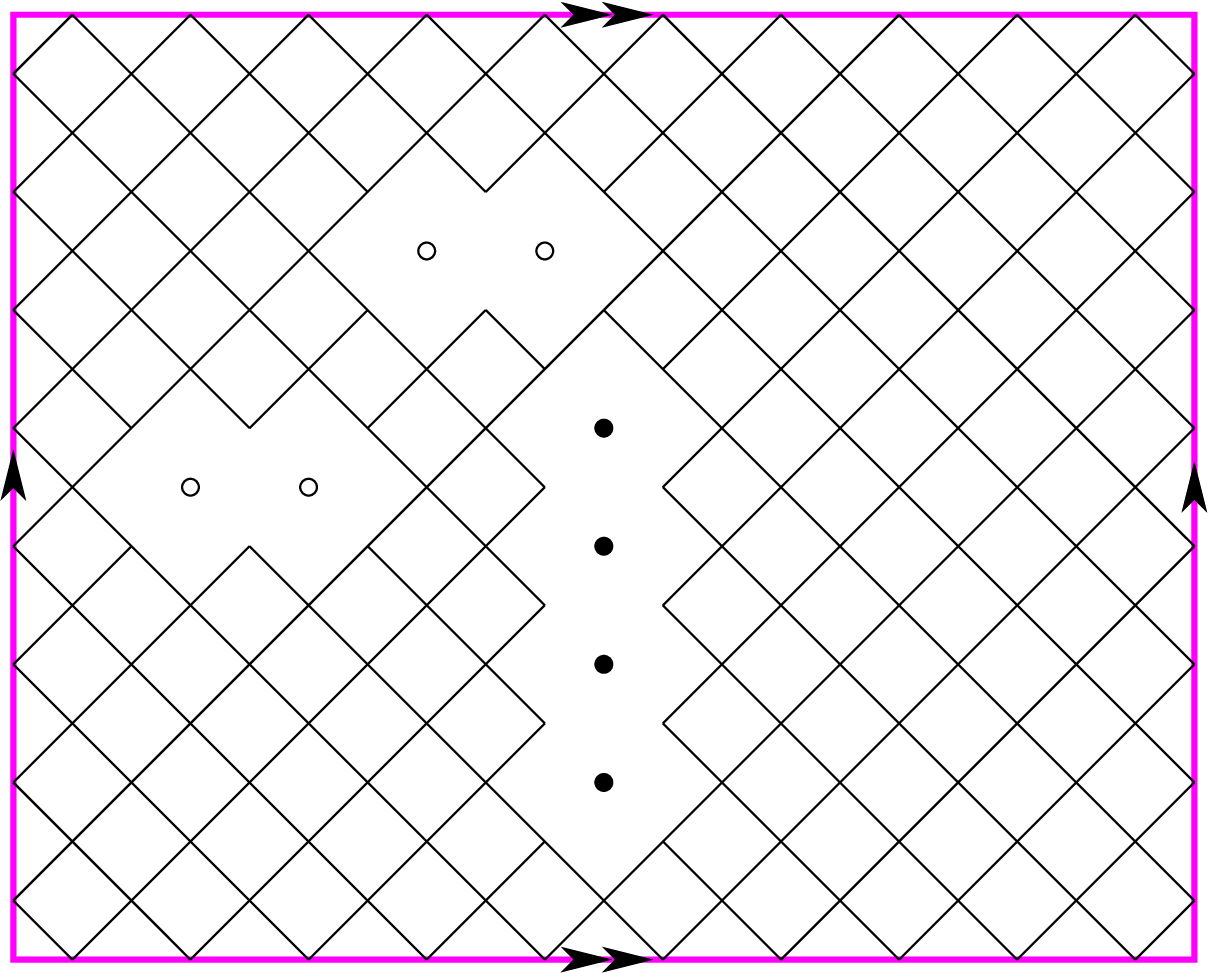}}
\hfill
{\includegraphics[width=0.4\textwidth]{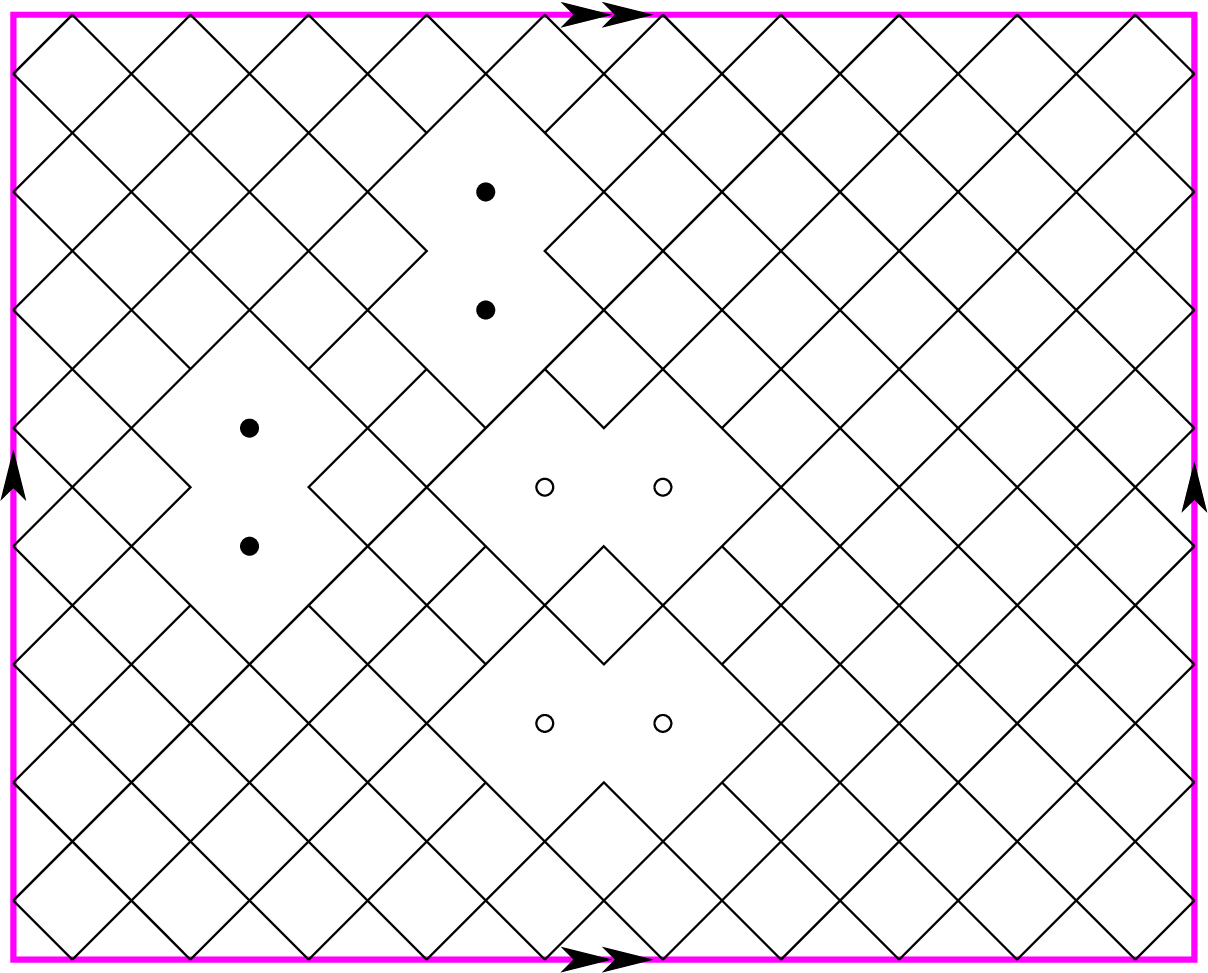}}
\hfill
}
\vskip0.1in
\caption{ The toroidal Aztec rectangle $T_{8,10}$ with four diagonal slits of length 2 (left; the two black slits happen to be contiguous in this example), and the same toroidal Aztec rectangle but with the diagonal slits flipped (right). By Theorem \ref{tcd}, they have the same number of perfect matchings.} 
\vskip-0.1in
\label{fcb}
\end{figure}

\begin{proof}
Without loss of generality assume that $O_1,\dotsc,O_s$ are black-placed vertical $k$-multiplets, and $O_{s+1},\dotsc,O_{2s}$ are white-placed horizontal $k$-multiplets. Then the former are black-placed $O_{k-1,0}$-shaped holes, while the latter are white-placed $O_{0,k-1}$-shaped holes. By definition, the shape of $\h(O_{k-1,0})$ is $O_{0,k-1}$, and the shape of $\ve(O_{0,k-1})$ is $O_{k-1,0}$. Therefore the statement follows directly from Theorem \ref{tcc} (as for $l=0$, the exponent of 2 in \eqref{ecc} becomes zero).
\end{proof}

We define the {\it finite size correlation} $\omega_{m,n}(O_1,\dotsc,O_k)$ of a collection of gaps $O_1,\dotsc,O_k$ of total charge zero on the square grid by encasing them in a toroidal Aztec rectangle graph $T_{m,n}$, and taking the ratio between the number of perfect matchings of the toroidal graph with holes and the intact toroidal graph:
\begin{equation}
\omega_{m,n}(O_1,\dotsc,O_k)=\frac{\M(T_{m,n}\setminus O_1\cup\cdots\cup O_k)}{\M(T_{m,n})}.
\label{ecex}
\end{equation}
Their {\it correlation} $\omega(O_1,\dotsc,O_k)$ is defined by setting $m=n$ and letting $n$ approach infinity:
\begin{equation}
\omega(O_1,\dotsc,O_k)=\lim_{n\to\infty}\frac{\M(T_{n,n}\setminus O_1\cup\cdots\cup O_k)}{\M(T_{n,n})}.
\label{ecey}
\end{equation}
We can extend this definition of $\omega$ to collections of holes $O_1,\dotsc,O_k$ of arbitrary total charge just as we did in \cite[Section 2]{gd}.

\medskip
\parindent0pt
{\it Remark $6$.}
The statement of Theorem \ref{tcd} is an unexpected symmetry of the correlation of diagonal slits, which holds already at the level of the {\it finite size correlation} of the gaps (and not just for their correlation)! There seems to be no intuitive reason why such a configuration of holes should have the same correlation as its flip. Indeed, one special case is when the horizontal, white-placed multiplets form some arbitrary ``profile'' recorded by some discrete function $f$, and the vertical, black-placed ones form some other arbitrary profile given by a function $g$. After flipping all these slits, the relative positions of the horizontal ones are related by $g$, and the relative positions of the vertical ones are related by $f$. As $f$ and $g$ are arbitrary, it is not at all clear why the two correlations should be equal.

\parindent15pt
For instance, it follows that the correlation of the four diagonal slits on the left in Figure~\ref{fcb} is the same as the correlation of the four on the right, even though geometrically the two configurations look quite different. This is confirmed also by a direct calculation of the two correlations (using the approach from \cite[Section 4.6]{Kenone}), both of which turn out to equal
\begin{equation*}
\frac{(\pi - 4) (54 \pi^2  - 339 \pi + 532)}{54\pi^4}.
\end{equation*}
%

%-> Show here the very case we verified using determinants of P-functions. Include also the common value of their correlation. State perhaps also that technically speaking our correlation is defined using a bit different toroidal conditions, it should lead to the same values --- and in deed this is confirmed by the observed equality of the two calculated correlations!!

Note also that the special case $k=2$ is reminiscent of something we discovered in our earlier work \cite{gd}. Namely, \cite[Equation (6.1)]{gd} gives a simple product formula for the correlation of doublets, defined there to be either pairs of consecutive monomers along a fixed diagonal $\ell$ of the square grid, or pairs of consecutive separation defects on that same diagonal $\ell$ (the latter play the role of pairs of neighboring monomers having color opposite to the color of the vertices on $\ell$, in a chessboard coloring of the vertices of the square grid; see \cite[Figure 1.1]{gd} for their definition). One readily verifies that this formula is invariant under swapping the type of the doublets (this amounts to the fact that the 2D Coulomb energy \cite[Equation (4.1)]{gd} of a system consisting of $n$ positive and $n$ negative unit charges is invariant under changing the signs of all charges). The flip-invariance of the correlation implied by Theorem \ref{tcd} discussed above holds in a much more general setting, as the diagonal multiplets can be placed anywhere in the plane (as opposed to being required to be on a common diagonal of the square grid, as in \cite{gd}).

\medskip
Under the assumption that the electrostatic conjecture (see \cite[Equation (2.12)]{gd}, and also \cite[Conjecture 1]{ov}) is true\footnote{ Technically speaking, the definition \eqref{ecey} is slightly different from the definition of the correlation $\bar\omega$ of \cite{gd}, as the toroidal Aztec diamonds graphs $T_{m,n}$ are different from the toroidal grid graphs considered in \cite{gd}. However, one can derive that $\omega=\bar\omega$ from the assumption that the electrostatic conjecture holds (see Remark~8).}, we obtain an explicit formula for the correlation of the vertices that form an odd Aztec rectangle $O_{k,l}$.

\begin{theo}
\label{tce}

\begin{equation}
\omega(O_{k,l})=2^{-\frac {kl}{2}} \omega(O_{0,k+l}).
\label{ece}
\end{equation}

\end{theo}

\begin{proof}
Setting $s=1$ in Theorem \ref{tcc} we obtain that for any white-placed $O_{k,l}$ and black-placed $O_{l.k}$ we have
\begin{equation}
\M(T_{m,n}\setminus O_{k,l}\cup O_{l,k}) = 2^{-kl} \M(T_{m,n}\setminus O_{0,k+l}\cup O_{k+l,0}).
\label{ecee}
\end{equation}
Dividing through by $\M(T_{m,n})$ yields
\begin{equation}
\omega_{m,n}(O_{k,l},O_{l,k}) = 2^{-kl} \omega_{m,n}(O_{0,k+l}, O_{k+l,0}).
\label{eceee}
\end{equation}
In turn, setting $m=n$ and sending $n$ to infinity, this gives
\begin{equation}
\omega(O_{k,l},O_{l,k}) = 2^{-kl} \omega(O_{0,k+l}, O_{k+l,0}).
\label{eceeee}
\end{equation}
The electrostatic conjecture implies that the left hand side has asymptotics $[\omega(O_{k,l})]^2\de^{-(k+l+1)^2/2}$ and the right hand side has asymptotics $[\omega(O_{0,k+l})]^2\de^{-(k+l+1)^2/2}$, where $\de$ is the distance between the centers of the considered pairs of holes. This proves
%the first equality in
\eqref{ece}.
%The second equality follows by \cite[Proposition 6.1]{gd}.
\end{proof}

%\medskip
In particular, the correlation of the odd Aztec diamond $OD_n=O_{n,n}$ --- which is a measure of how many tilings its exterior has --- satisfies the following relation, which can be regarded as a dual of the Aztec diamond theorem of \cite{EKLP} (since the latter can be stated as $\M(AD_n)=2^{\frac{n(n+1)}{2}}\M(AD_0)$).

\begin{cor}
\label{tcf}

\begin{equation}
\omega(O_{n,n})=2^{-\frac {n^2}{2}} \omega(O_{0,2n}).
\label{ecf}
\end{equation}

\end{cor}

\parindent0pt
The exact expression for $\omega(O_{0,m})$ is derived in Remark 8 below (see equations \eqref{eceea}--\eqref{eceeb}).

\medskip
\parindent0pt
{\it Remark $7$.}
Theorem \ref{tce} shows that in the sequence of odd Aztec rectangles $O_{0,n},O_{1,n-1},\dotsc,O_{n,0}$ in which the sum of the width and height is constant, the individual correlations of the sets of monomers they consist of are all the same up to a power of two. It also follows that the correlation is minimum (i.e., it is hardest to accommodate the corresponding structure in a dimer covering of the infinite square grid) in the case when the width and height are equal (for $n$ even) or nearly equal (i.e., one unit apart, for $n$ odd). Both these facts seem highly non-trivial, yet, as we will see from our proofs, they follow in a simple fashion by repeated application of the complementation theorem \cite[Theorem 2.1]{CT}.

\medskip
\parindent0pt
{\it Remark $8$.} The conjectural exact value of $\omega(O_{k,l})$ can be derived as follows. By Theorem \ref{tce}, it is enough to derive the value of $\omega(\underbrace{\circ\cdots\circ}_{m})$ of the correlation of a diagonal multiplet consisting of $m$ monomers. In \cite[Proposition 6.1]{gd} we proved that for the correlation $\bar{\omega}$ we were working with there (see \cite[Section 2]{gd} for its definition), we have
\begin{align}
\bar\omega(\underbrace{\circ\cdots\circ}_{2i})&=
\frac{2^{i(i-1)/2}}{\pi^i}0!\,1!\cdots(2i-1)!
\label{eceea}
\\[5pt]
\bar\omega(\underbrace{\circ\cdots\circ}_{2i+1})&=
2^{i^2/2}\bar\omega(\circ)\, 0!\,1!\cdots(2i)!,
\label{eceeaa}
\end{align}
where, if $A= 1.28242712...$ is the Glaisher-Kinkelin constant\footnote{ The Glaisher-Kinkelin constant  $A$ (see \cite{Glaish}) is defined by the limit 
$\lim_{n\to\infty}
\dfrac
 {0!\,1!\,\cdots\,(n-1)!}
 {n^{\frac{n^2}{2}-\frac{1}{12}}\,(2\pi)^{\frac{n}{2}}\,e^{-\frac{3n^2}{4}}}
=
\dfrac
 {e^{\frac{1}{12}}}
 {A}
$.}, we have
\begin{equation}
\bar\omega(\circ)=\frac{e^{\frac14}}{2^{\frac{7}{24}}A^3}.
\label{eceeb}
\end{equation}
The electrostatic conjecture for our correlation $\omega$ (which should hold, as $\omega$ is a  natural square lattice analog of the correlation in \cite[Conjecture 1]{ov}) implies that\footnote{ Here $\de$ denotes the Euclidean distance.}
\begin{equation}
\omega(\circ,\bullet)\sim \omega(\circ)\,\omega(\bullet)\de(\circ,\bullet)^{-1/2},
\label{eceec}
\end{equation}
in the limit of large separation between a white monomer $\circ$ on a lattice diagonal $\ell$ and a black monomer $\bullet$ on the diagonal $\ell'$ just below $\ell$. On the other hand, a special case of \cite[Theorem 3.1]{gd} implies that for the correlation $\bar\omega$ defined in \cite[Section 2]{gd} we have
\begin{equation}
\bar\omega(\circ,\bullet)\sim \bar\omega(\circ)\,\bar\omega(\bullet)\de(\circ,\bullet)^{-1/2}.
\label{eceed}
\end{equation}
However, the correlations $\omega$ and $\bar\omega$  should only differ by a multiplicative constant (coming from the fact that the two definitions use different normalizations), as both are defined by including the holes in large regions --- growing to infinity --- in locations whose dimer statistics is not distorted in the limit. Thus we should have $\omega(\circ)=c\,\bar\omega(\circ)$, for some constant $c>0$. But since by translation invariance $\omega(\circ)=\omega(\bullet)$ and $\bar\omega(\circ)=\bar\omega(\bullet)$, equations \eqref{eceec} and \eqref{eceed} imply $c=1$, which shows that $\omega$ is in fact equal to $\bar\omega$. Thus formulas \eqref{eceea} and \eqref{eceeaa} hold for $\omega$ in place of $\bar\omega$ as well.

\parindent15pt

%\medskip
%In particular, the correlation of the odd Aztec diamond $OD_n=O_{n,n}$ --- which is a measure of how many tilings its exterior has --- has the following exact expression, which can be regarded as a dual of the Aztec diamond theorem of \cite{EKLP} (since the latter can be stated as $\M(AD_n)=2^{\frac{n(n+1)}{2}}\M(AD_0)$).

%\begin{cor}
%\label{tcf}

%
%\begin{equation}
%\omega(O_{n,n})=2^{-\frac {n^2}{2}} \omega(O_{0,2n})=...
%\label{ecf}
%\end{equation}
%

%\end{cor}

\section{Proof of Theorems \ref{tba}--\ref{tbd}}

For all these proofs it will be convenient to phrase the statements in the equivalent language of perfect matchings of the graphs that are the planar duals of the regions involved in Theorems~\ref{tba}--\ref{tbd}. This is because the key to our proofs is to apply the complementation theorem for cellular graphs we obtained in \cite{CT} (see Theorem 2.1 there).

We now recall this theorem. We use for this purpose the graph $H$ indicated in Figure \ref{fda} by the thick black contours.
%This is the graph on the right hand side of \eqref{tba} --- the planar dual of the Aztec rectangle with odd windows $AR_{m,m+a}\setminus O_{0,a_1}\cup\cdots\cup O_{0,a_s}$ (for simplicity, we drop from our notation the superscripts indicating the positions of the centers of the holes), in the special case when $m=6$, $s=1$ and $a_1=3$ (so $a=\sum_{i=1}^s(a_i+1)=4$).
%We now recall the complementation theorem and explain how to apply it to $H$.

\begin{figure}[t]
\centerline{
\hfill
{\includegraphics[width=0.40\textwidth]{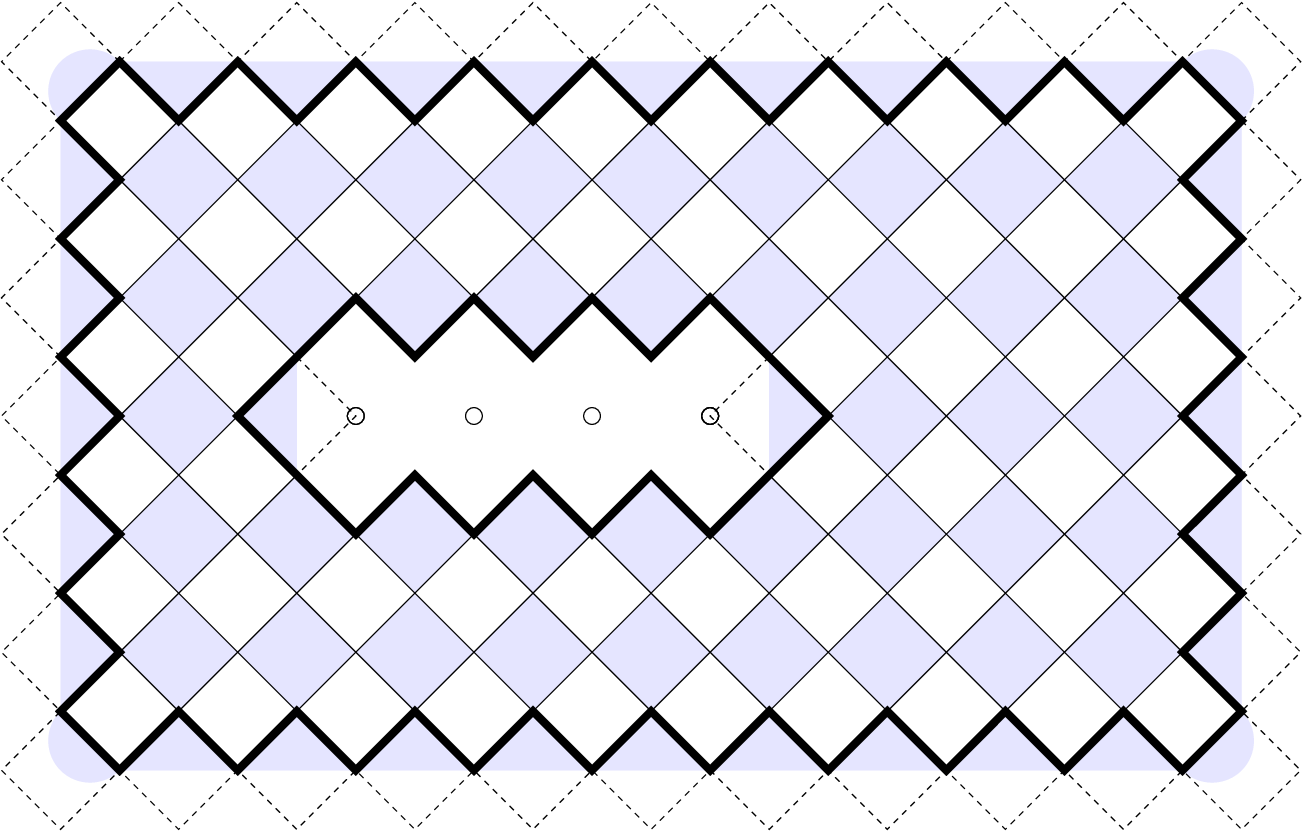}}
\hfill
{\includegraphics[width=0.40\textwidth]{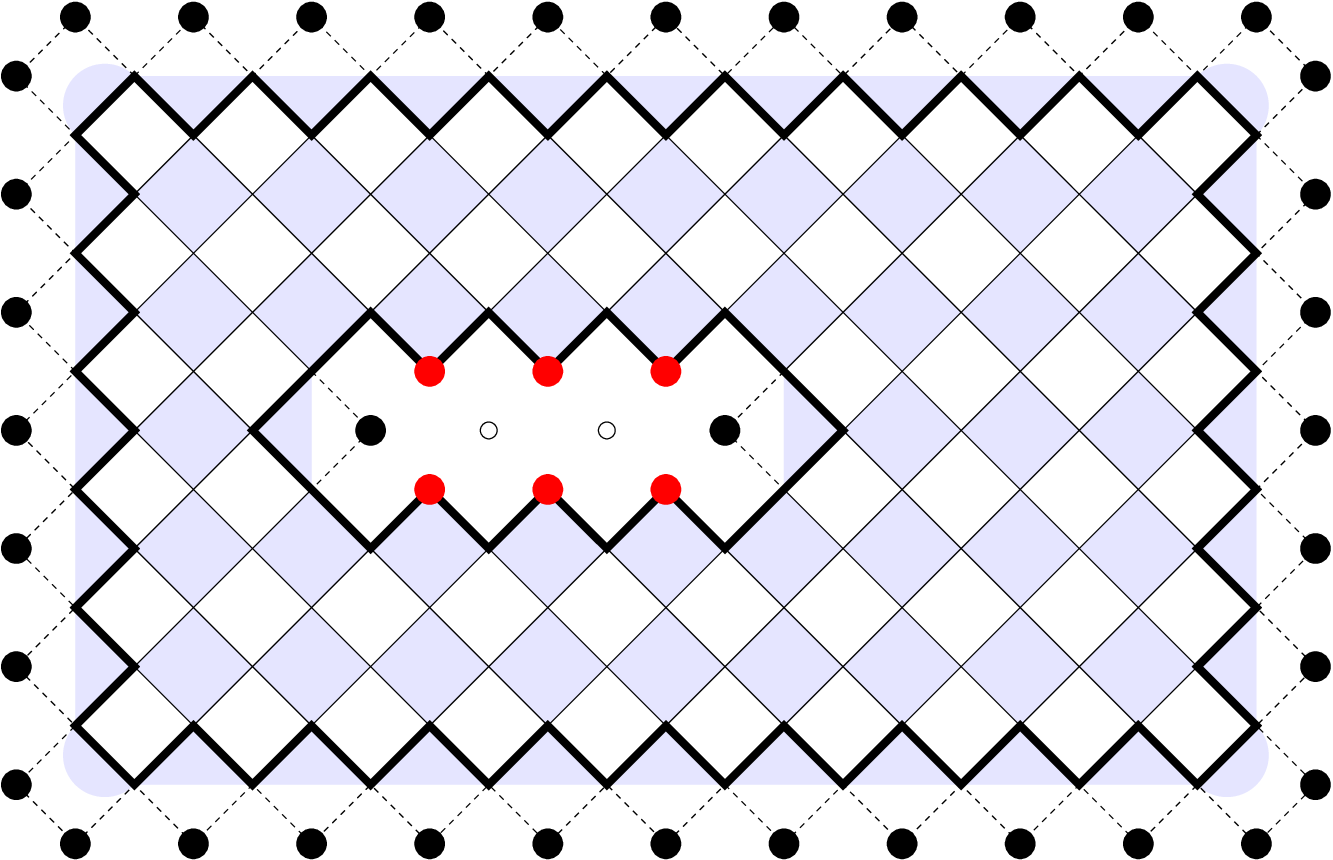}}
\hfill
}
\vskip0.1in
\caption{ {\it Left.} A graph $H$ (the portion of the grid graph contained between the solid black contours) and its cellular completion $G$ (the graph obtained by  adding the dotted lines); the cells of $G$ are shaded (completely if all 4 edges of the cell are contained in $H$, partially otherwise).  {\it Right.} The extremal vertices of $G$ that are in $H$ are indicated by red dots, and those that are not in $H$ by black dots.}
\vskip-0.1in
\label{fda}
\end{figure}

\begin{figure}[t]
\centerline{
%\hfill
{\includegraphics[width=0.40\textwidth]{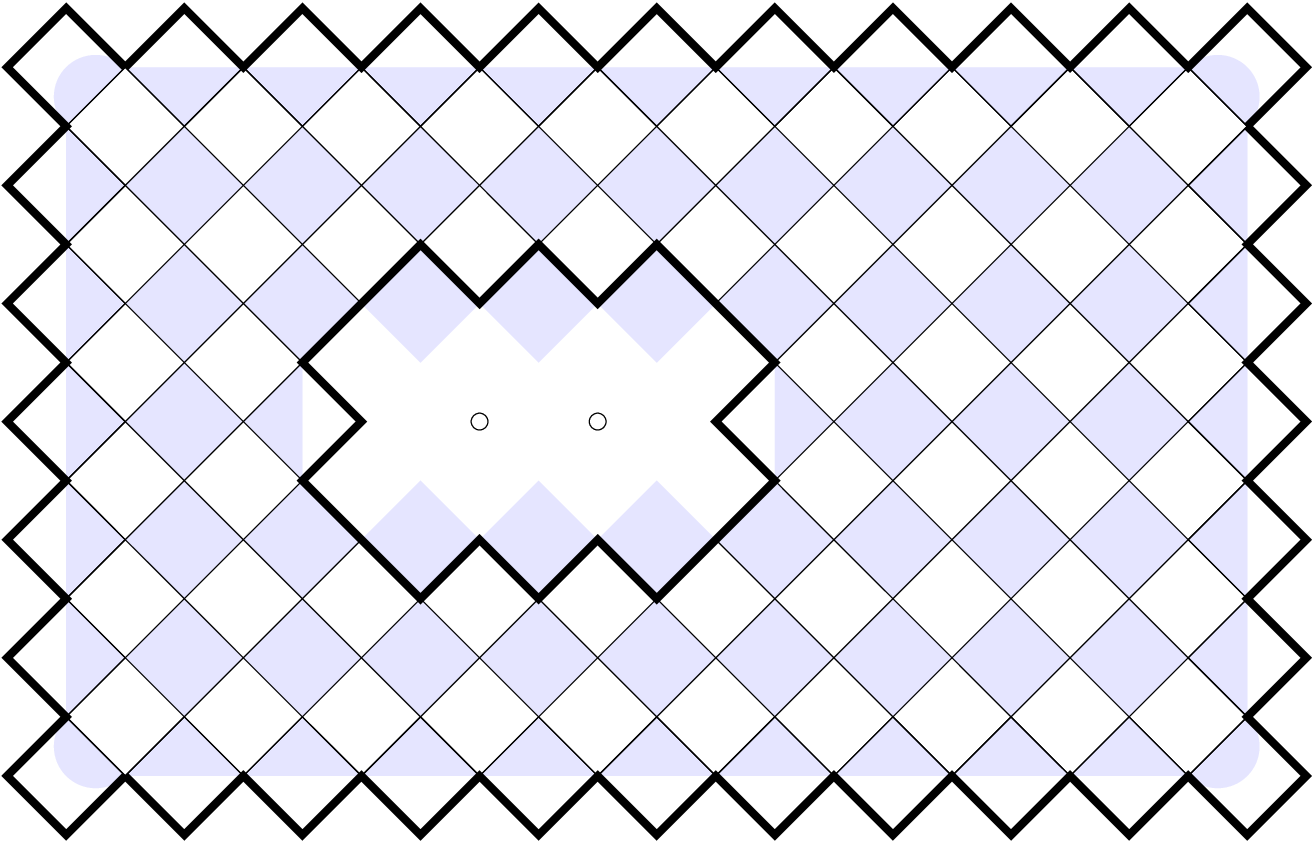}}
%\hfill
}
\vskip0.1in
\caption{ The complement $H'$ (the portion contained between the solid black contours) of $H$ with respect to its cellular completion $G$ is the induced subgraph of $G$ whose vertex set is obtained from the vertex set of $H$ by discarding all extremal vertices of $G$ that are in $H$, and including all of them that are not in $H$.}
\vskip-0.1in
\label{fdb}
\end{figure}

Color the unit squares of the square grid in chessboard fashion as indicated on the left in Figure~\ref{fda}, and consider the shaded unit squares that contain at least one edge of $H$ (in the figure, the instances when not all the four edges of the shaded unit square belong to $H$ are shown as only partially shaded). Let $G$ be the graph consisting of the union of all 4-cycles of the grid which enclose a (completely or partially) shaded face.

Then $G$ is an instance of a {\it cellular graph}. Its shaded 4-cycles are called {\it cells}. The cells can be stringed up along grid diagonals (horizontal or vertical) into maximal contiguous runs, called {\it paths of cells}. Each of the two vertices in a path that are furthest apart from one another is called an {\it extremal vertex} of $G$; the set of extremal vertices of $G$ is denoted by $X(G)$. Note that $(i)$ $H$ is an induced subgraph of $G$, and $(ii)$ all the vertices of $G$ that are not in $H$ are extremal vertices. In such instances we say that $G$ is a {\it cellular completion of $H$}.

The {\it complement of $H$} with respect to its cellular completion $G$ is the induced subgraph $H'$ of $G$ whose vertex set is determined by the equation\footnote{ $V(G)$ denotes the set of vertices of the graph $G$.} $V(H')\triangle V(H)=X(G)$. In other words, $V(H')$ is the set obtained from $V(H)$ after performing the following operation at each end of every path of cells of $G$: if the corresponding extremal vertex belongs to $V(H)$, remove it; otherwise, include it.

For our graph $H$, this construction is illustrated in Figures \ref{fda} and \ref{fdb}. The red dots in the picture on the right in Figure \ref{fda} indicate the extremal vertices of $G$ that are in $H$, while the black dots mark the ones that are not in $H$. The complement $H'$ of $H$ is the portion of the grid graph enclosed between the thick black contours in Figure \ref{fdb}.

The complementation theorem states that the number of perfect matchings of $H$ and the number of perfect matchings of its complement $H'$ differ only by a power of 2. More precisely, define the {\it type} $\tau(L)$ of a path of cells $L$ to be 1 less than the number of closed ends of $L$ (an end of~$L$ is said to be closed if $H$ contains the corresponding extremal vertex of $G$). Then we have the following relationship between the numbers of perfect matchings of $H$ and $H'$.

\begin{theo} \cite[Theorem 2.1 (Complementation Theorem)]{FT}
Let $G$ be a cellular graph with its set of cells partitioned into disjoint paths $L_{1},L_{2},\dotsc,L_{k}$. If $G$ is a cellular 
completion of the subgraph $H$, and $H'$ is the complement of $H$ with respect to $G$, we have
\begin{equation}
\M(H)=2^{t}\M(H'),
\label{eda}
\end{equation}
where the exponent of $2$ is $t=\tau(L_{1})+\tau(L_{2})+\cdots+\tau(L_{k})$.
\end{theo}

Let us apply this theorem to our graph $H$ shown in Figure \ref{fda}. 
Consider the partition of the set of cells of $G$ into horizontal paths. All the paths are contained in the 7 rows of 4-cycles of the grid that contain some shaded (or partially shaded) cells. The top 3 of these rows contain one path of cells each, and each of them is of type $-1$; the same holds for the bottom 3 rows. The remaining  (middle) row contains 2 paths of cells, each of type $-1$ as well. Therefore, in \eqref{eda} we have $t=6\cdot(-1)+2\cdot(-1)$, and we obtain
\begin{equation}
\M(H)=2^{-8}\M(H').
\label{edbb}
\end{equation}

{\it Proof of Theorem $\ref{tba}$.} The above example generalizes. Let $H$ be the planar dual graph of the Aztec rectangle region with odd windows $AR_{m,m+a}\setminus O_{0,a_1}\cup\cdots\cup O_{0,a_s}$ (for simplicity, we drop from our notation the superscripts indicating the positions of the centers of the holes along the symmetry axis $\ell$). Apply the complementation theorem to $H$.

We can use Figures \ref{fda} and \ref{fdb} to illustrate the general situation; they correspond to the special case when $m=6$, $s=1$ and $a_1=3$ (so $a=\sum_{i=1}^s(a_i+1)=4$).

The definition of the cellular completion of $G$ we used in our example above readily extends to the general case. Consider the partition of the set of cells of $G$ into horizontal paths. All the paths are contained in the $m+1$ rows of 4-cycles of the grid that contain some shaded (or partially shaded) cells. The top $m/2$ of these rows\footnote{ Recall that in the statement of Theorem \ref{tca} it is assumed that $m$ is even.} contain one path of cells each, and each of them is of type $-1$; the same holds for the bottom $m/2$ rows. The remaining  (middle) row contains $s+1$ paths of cells, each of type $-1$ as well. Therefore, in \eqref{eda} we have $t=m\cdot(-1)+(s+1)\cdot(-1)$, and we obtain
\begin{equation}
\M(H)=2^{-m-(s+1)}\M(H').
\label{edb}
\end{equation}
Since we will apply the complementation theorem repeatedly, we will find it convenient to show its application in a single figure. To this end, we record the information in Figures \ref{fda} and \ref{fdb} in the picture on the left in Figure \ref{fdc}: the graph indicated by the solid black contours is the graph $H$ to which we apply the complementation theorem, the shading indicates the cells of the cellular completion $G$ of $H$, and the graph indicated by the thick magenta contours is its resulting complement $H'$.

The crucial fact upon which our proof hinges is the way the outer and inner boundaries of our regions evolve under such applications of the complementation theorem. Namely, it is clear that the outer boundary of the resulting $H'$ is the Aztec rectangle $AR_{m+1,m+a+1}$; so the outer boundary of $H$, which is $AR_{m,m+a}$, evolves into $AR_{m+1,m+a+1}$ by applying the complementation theorem in this way. As far as the inner boundaries are concerned, by definition, the $i$th hole in $H$ is obtained by removing a copy of the odd Aztec rectangle graph $O_{0,a_i}$. On the other hand, one readily sees that by the application of the complementation theorem, the corresponding resulting hole in $H'$ is obtained by removing a copy of the odd Aztec rectangle graph $O_{1,a_i-1}$, placed so that it has the same center as the copy of $O_{0,a_i}$ removed from $H$. We summarize this by saying that, under this application of the complementation theorem, the boundary of the $i$th hole evolved from having shape $O_{0,a_i}$ to having shape $O_{1,a_i-1}$. It follows that $H'$ is precisely the Aztec rectangle with odd windows $AR_{m+1,m+a+1}\setminus O_{1,a_1-1}\cup\cdots\cup O_{1,a_s-1}$, with holes placed so that their centers agree with the centers of the corresponding holes in $H$.

\begin{figure}[t]
\centerline{
\hfill
{\includegraphics[width=0.40\textwidth]{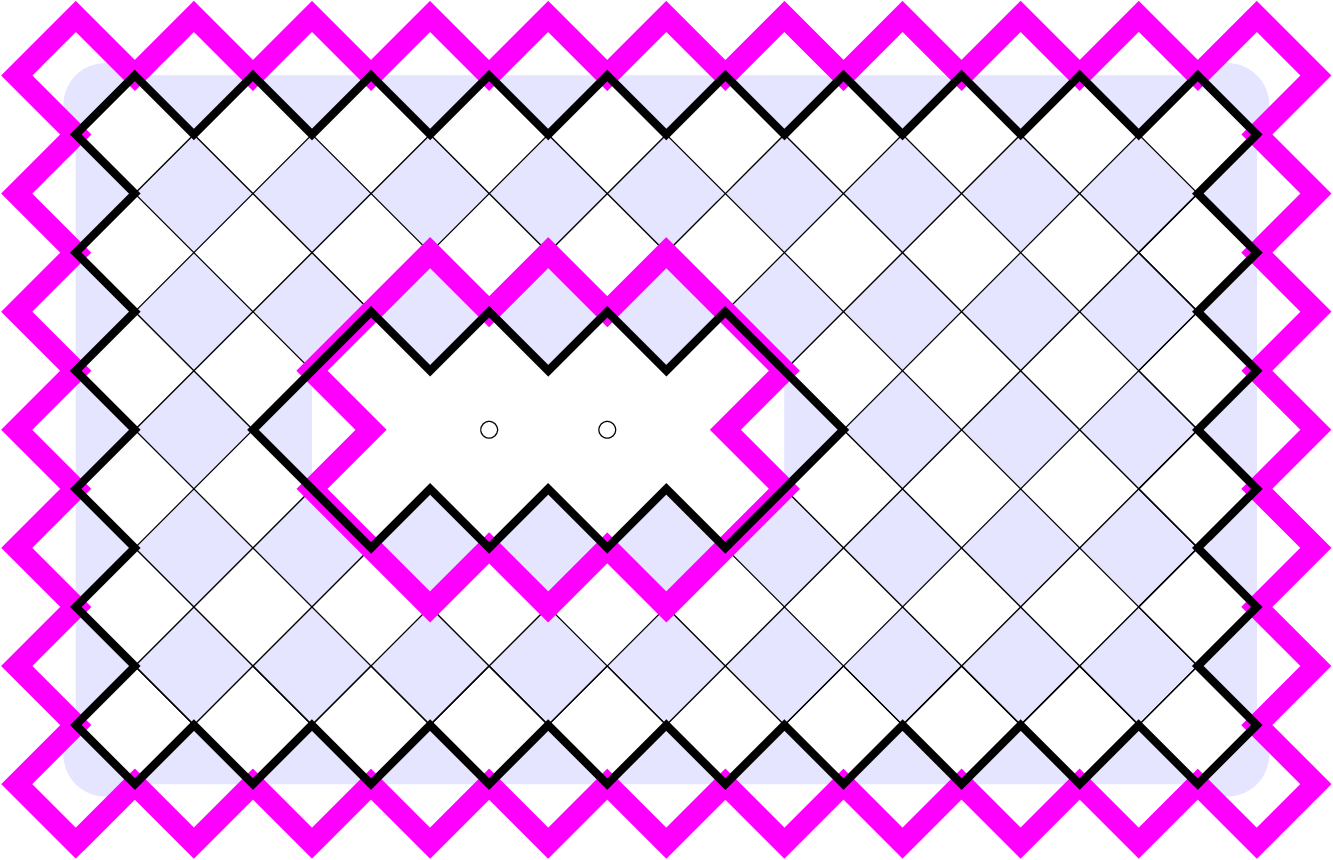}}
\hfill
{\includegraphics[width=0.40\textwidth]{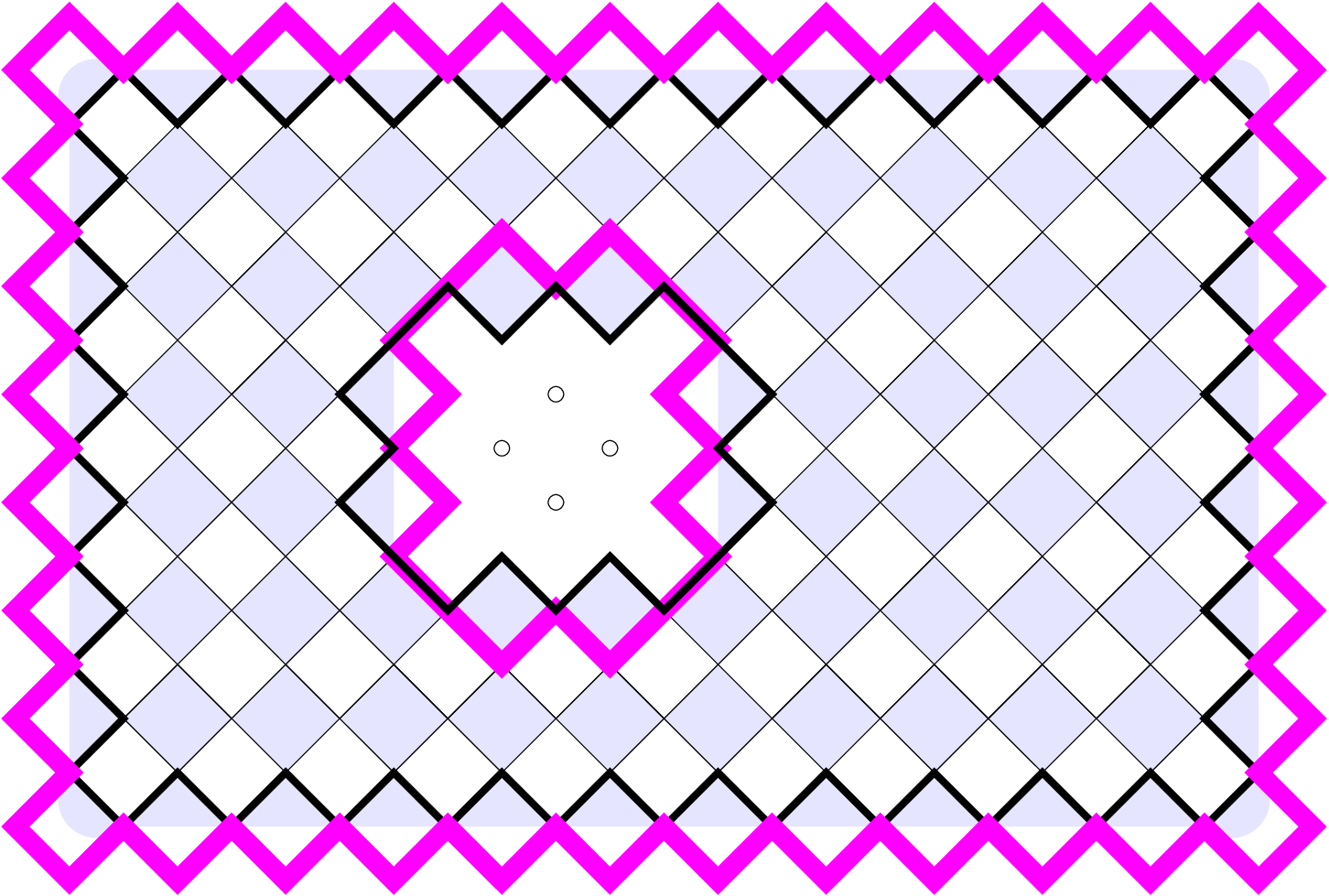}}
\hfill
}
\vskip0.1in
\caption{ {\it Left.} Applying the complementation theorem to the graph $H=AR_{m,m+a}\setminus O_{0,a_1}\cup\cdots\cup O_{0,a_s}$ (here $m=6$, $s=1$ and $a_1=3$, so that $a=\sum_{i=1}^s(a_i+1)=4$); $H$ is the portion of the grid graph between the solid black contours. The resulting complement $H'$ (the portion between the thick magenta contours) is the graph $AR_{m+1,m+a+1}\setminus O_{1,a_1-1}\cup\cdots\cup O_{1,a_s-1}$.  {\it Right.} The next application of the complementation theorem turns $AR_{m+1,m+a+1}\setminus O_{1,a_1-1}\cup\cdots\cup O_{1,a_s-1}$ into $AR_{m+2,m+a+2}\setminus O_{2,a_1-2}\cup\cdots\cup O_{2,a_s-2}$.}
\vskip-0.1in
\label{fdc}
\end{figure}

\begin{figure}[t]
\centerline{
%\hfill
{\includegraphics[width=0.40\textwidth]{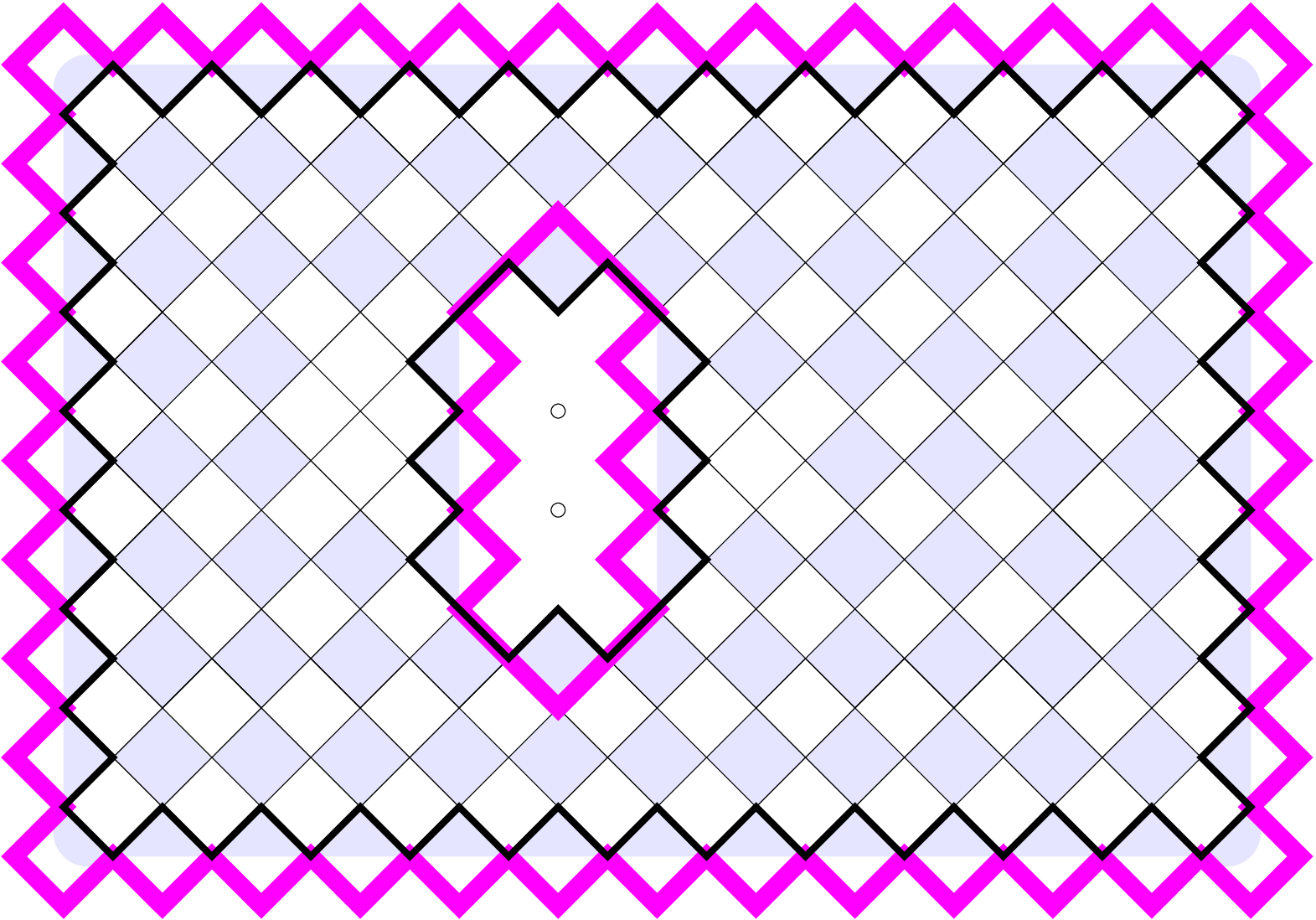}}
%\hfill
}
\vskip0.1in
\caption{ One more application of the complementation theorem, turning the graph $AR_{m+2,m+a+2}\setminus O_{2,a_1-2}\cup\cdots\cup O_{2,a_s-2}$ which resulted on the right in Figure \ref{fdc} into $AR_{m+3,m+a+3}\setminus O_{3,a_1-3}\cup\cdots\cup O_{3,a_s-3}$.}
\vskip-0.1in
\label{fdd}
\end{figure}

This pattern continues. Applying the complementation theorem next to $H=AR_{m+1,m+a+1}\setminus O_{1,a_1-1}\cup\cdots\cup O_{1,a_s-1}$ --- with the shading of the cells chosen so that some of the unit squares outside $H$ are partially shaded --- the resulting $H'$ is readily seen to be precisely the region $AR_{m+2,m+a+2}\setminus O_{2,a_1-2}\cup\cdots\cup O_{2,a_s-2}$. To obtain the exponent of 2, partition the set of cells as before into horizontal paths. There are again $m$ uninterrupted paths stretching across the graph, each of type $-1$. This time there are two rows of 4-cycles each of which contains $s+1$ shorter paths, all of type $-1$. Therefore, the exponent of 2 at this application of the complementation theorem is $m\cdot(-1)+2(s+1)\cdot(-1)=-m-2(s+1)$.

Continuing to apply the complementation theorem this way, at the $k$th application (provided $k\leq a_i$ for $i=1,\dotsc,s$, as we explain in the next paragraph) we are applying it to $H=AR_{m+k-1,m+a+k-1}\setminus O_{k-1,a_1-k+1}\cup\cdots\cup O_{k-1,a_s-k+1}$, the resulting $H'$ is $AR_{m+k,m+a+k}\setminus O_{k,a_1-k}\cup\cdots\cup O_{k,a_s-k}$, and the exponent of two is $m\cdot(-1)+k(s+1)\cdot(-1)=-m-k(s+1)$. Putting all this together we obtain that
\begin{align}
\M(AR_{m,m+a}\setminus O_{0,a_1}\cup\cdots\cup O_{0,a_s})
&=
2^{\sum_{j=1}^k(-m-j(s+1))}\M(AR_{m+k,m+a+k}\setminus O_{k,a_1-k}\cup\cdots\cup O_{k,a_s-k})
\nonumber
\\
&=
2^{-km-{k+1\choose 2}(s+1)}\M(AR_{m+k,m+a+k}\setminus O_{k,a_1-k}\cup\cdots\cup O_{k,a_s-k}),
\label{edc}
\end{align}
which proves \eqref{ebb}.

The reason we need $k\leq a_i$ for all $i$ is the following. Let $a_j$ be the minimum of the $a_i$'s. Then, after $a_j$ applications of the complementation theorem, the boundary of the $j$th hole has evolved into having shape $O_{a_j,0}$, and a further application of the complementation theorem --- which involves switching the shaded squares of the square grid for the ones that were not shaded at its previous application --- would result in a graph that is not embedded in the grid graph.
$\hfill\square$

\medskip
{\it Proof of Theorem $\ref{tbb}$.} There are two ways to choose the shading when we apply the complementation theorem to a subgraph of the grid graph, corresponding to the white or the black unit squares of the square grid when regarded as an infinite chessboard. For a subgraph $H$ whose outer boundary is that of an Aztec rectangle $AR_{m,n}$, one of these choices will make the boundary of $H'$ to be the larger Aztec rectangle $AR_{m+1,n+1}$, and the other will make it the smaller $AR_{m-1,n-1}$. If in addition $H$ has a hole in the shape of an odd Aztec rectangle $O_{k,l}$, for one shading the corresponding hole of $H'$ will evolve into the shape $O_{k+1,l-1}$ of bigger height and smaller width, while for the other the hole will evolve into the shape $O_{k-1,l+1}$ of smaller height and bigger width. One crucial fact in the proof of Theorem \ref{tba} above was that the holes were placed so that the shading that made the outer boundary become bigger also made the holes become taller and less wide. This fact holds true also for the planar dual graphs of the $AR'$ regions in Theorem \ref{tbb}, as one can readily see by comparing the two graphs on the left in Figure \ref{fbaa}.

Thanks to this, the same arguments as in the above proof apply ad litteram for the proof of Theorem \ref{tbb}. The only difference is that the holes are now placed symmetrically about $\ell'$ --- the translation by one unit in the southeast direction of the horizontal symmetry axis $\ell$ of the outer boundary. Since with each application of the complementation theorem this description of the placement of the evolving boundaries remains valid, this relates the number of perfect matchings of the $R'$-graph on the right hand side of \eqref{ebd} to that of the planar dual of the region on the left hand side of \eqref{ebd}, and the two are related by the same multiplicative constant as in \eqref{ebb}.
$\hfill\square$

\medskip
{\it Proof of Theorem $\ref{tbc}$.} The same approach as in the above two proofs works. Again, we start from an $R''$-graph and apply to it repeatedly the complementation theorem in such a way that at each application the height parameter of the odd Aztec rectangular holes (which is 0 to start with) increases one unit. The main difference compared to the above two proofs (and this will be the same for the proof of Theorem \ref{tbd}) is that the shading of the cells that achieves this has the effect of evolving the outer boundary into {\it smaller} Aztec rectangles, instead of larger ones. This can readily be seen in Figure \ref{fde}.

\begin{figure}[t]
\centerline{
%\hfill
{\includegraphics[width=0.30\textwidth]{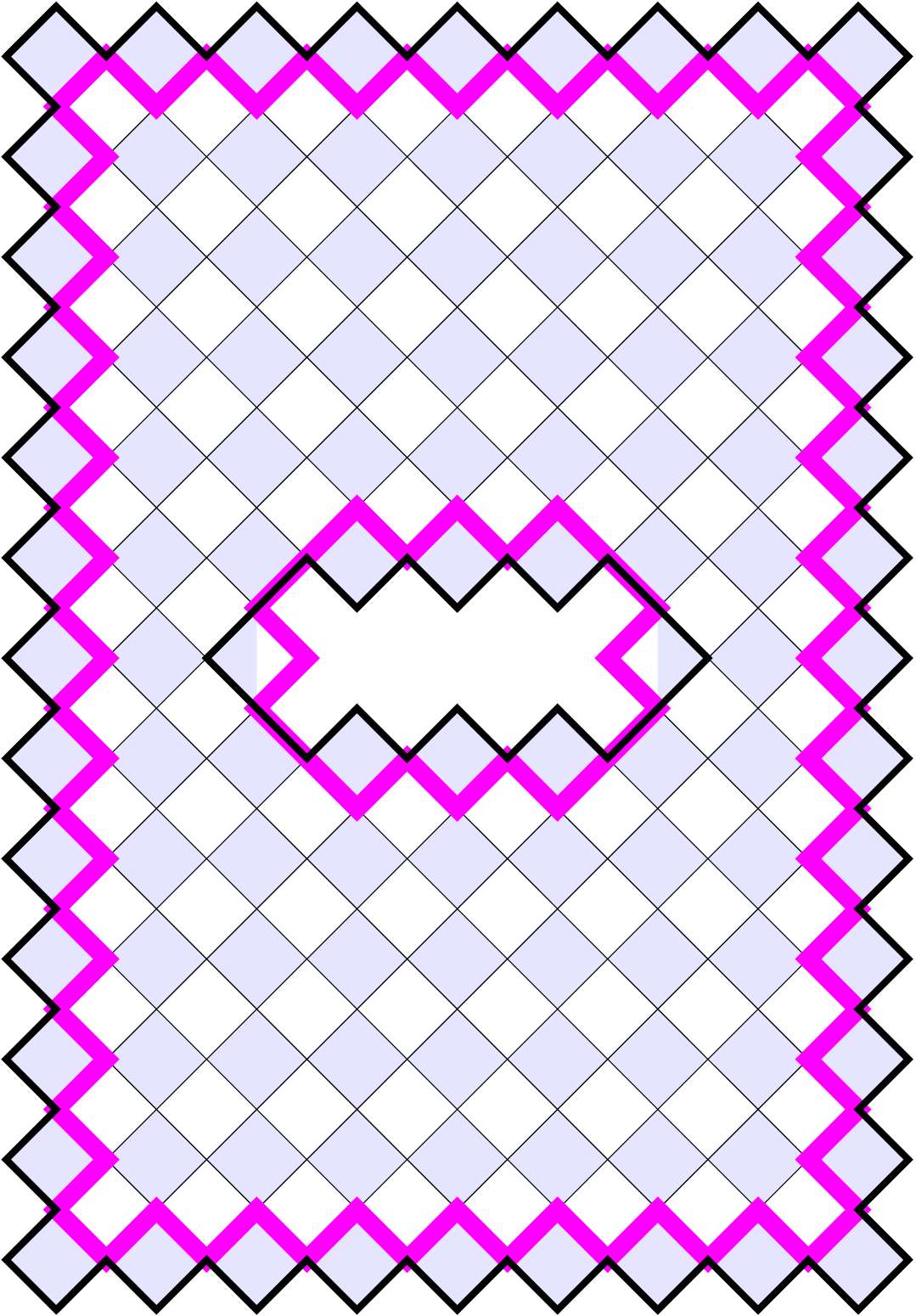}}
\hfill
{\includegraphics[width=0.30\textwidth]{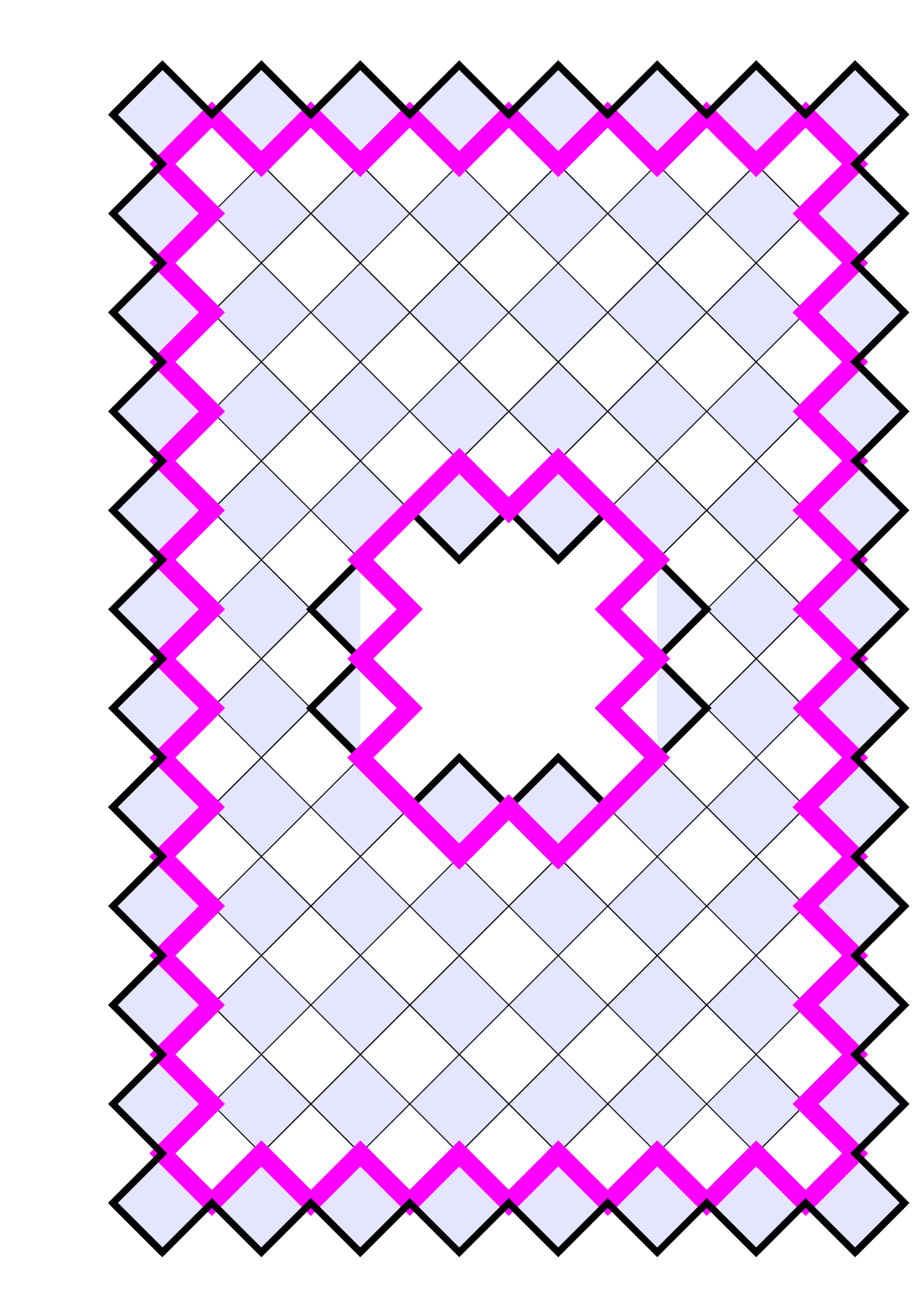}}
\hfill
{\includegraphics[width=0.30\textwidth]{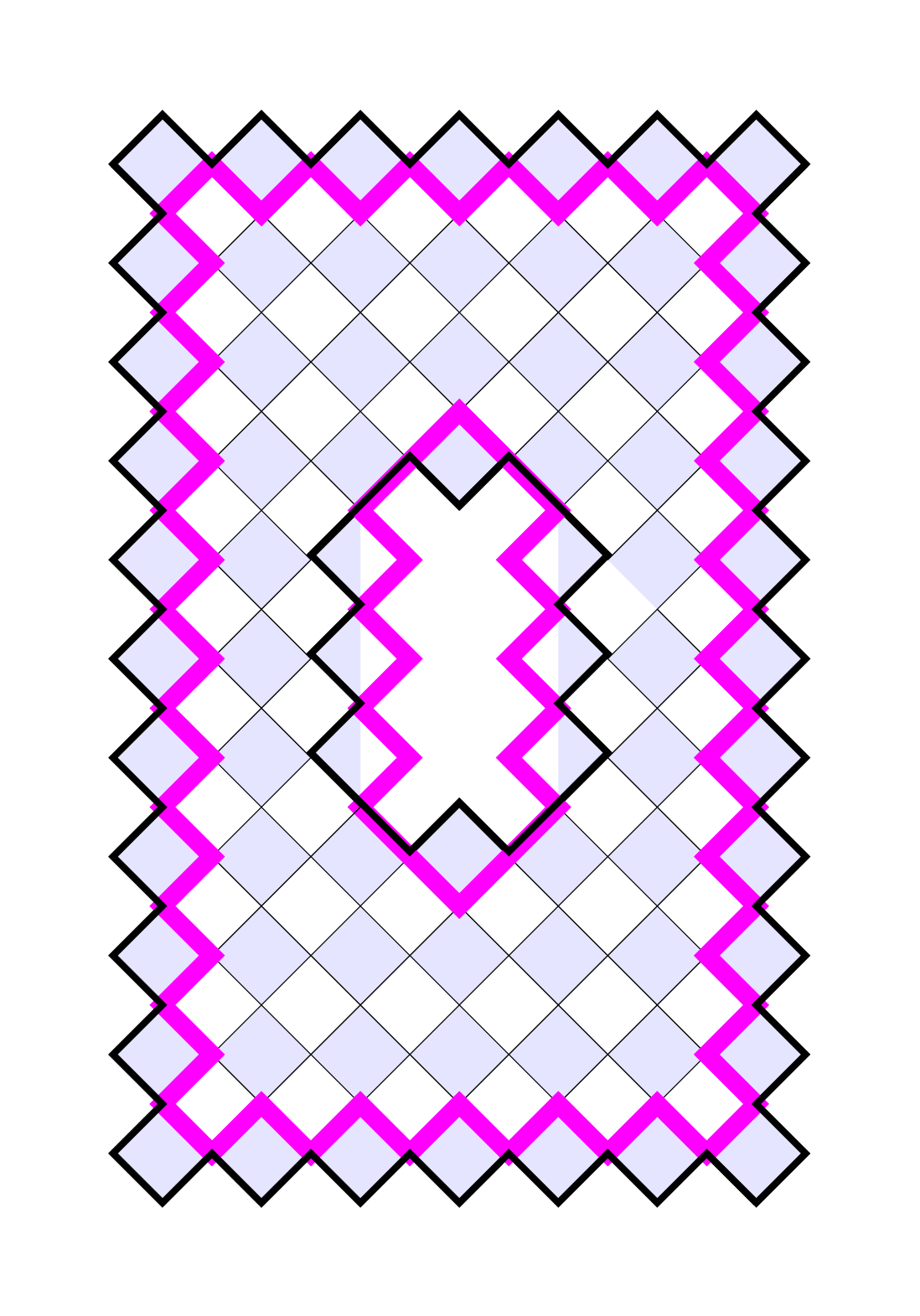}}
%\hfill
}
\vskip0.1in
\caption{ {\it Left.} Applying the complementation theorem to the graph $H=AR''_{m,m-a}\setminus O_{0,a_1}\cup\cdots\cup O_{0,a_s}$ (here $m=13$, $s=1$ and $a_1=3$, so that $a=\sum_{i=1}^s(a_i+1)=4$); $H$ is the portion of the grid graph between the solid black contours. The resulting complement $H'$ (the portion between the thick magenta contours) is the graph $AR''_{m-1,m-a-1}\setminus O_{1,a_1-1}\cup\cdots\cup O_{1,a_s-1}$.  {\it Center.} The next application of the complementation theorem turns $AR''_{m-1,m-a-1}\setminus O_{1,a_1-1}\cup\cdots\cup O_{1,a_s-1}$ into $AR''_{m-2,m-a-2}\setminus O_{2,a_1-2}\cup\cdots\cup O_{2,a_s-2}$. {\it Right.} One more application of the complementation theorem, turning $AR''_{m-2,m-a-2}\setminus O_{2,a_1-2}\cup\cdots\cup O_{2,a_s-2}$ into $AR''_{m-3,m-a-3}\setminus O_{3,a_1-3}\cup\cdots\cup O_{3,a_s-3}$.}
\vskip-0.1in
\label{fde}
\end{figure}

Let $H$ be the planar dual graph of the Aztec rectangle with odd windows $AR''_{m,m-a}\setminus O_{0,a_1}\cup\cdots\cup O_{0,a_s}$ (the picture on the left in Figure \ref{fde} illustrates the case $m=13$, $s=1$ and $a_1=3$, so that $a=\sum_{i=1}^s(a_i+1)=4$); again, for notational simplicity, we omit the upper indices that specify the location along the horizontal symmetry axis $\ell$ of the centers of the holes. Apply the complementation theorem to $H$ with respect to its cellular completion $G$ indicated by the shading in the picture on the left in Figure \ref{fde}. Then the complement $H'$ of $H$ with respect to $G$ is precisely the planar dual of $AR''_{m-1,m-a-1}\setminus O_{1,a_1-1}\cup\cdots\cup O_{1,a_s-1}$, with the centers of the new, evolved holes having the same location along $\ell$ as the old ones. To find the exponent of 2, partition again the set of cells of $G$ into horizontal paths. Out of the total of $m$ rows of 4-cycles that contain all the cells, $m-1$ (all but the central one) consist of a single path of cells, each of type 1.

For the analysis of the central row, assume first that none of the $s$ holes in $AR''_{m,m-a}\setminus O_{0,a_1}\cup\cdots\cup O_{0,a_s}$ communicates with the exterior. The central row is broken up by these $s$ holes into $s+1$ paths of cells: the two extreme ones of type 0, and the remaining $s-1$ of type $-1$. Thus the exponent $t$ of the power of 2 in the complementation theorem is $t=(m-1)-(s-1)=(m+1)-(s+1)$. Next, notice that if for instance the leftmost hole does communicate with the exterior\footnote{ The case when the rightmost hole communicates with the exterior is handled in the same fashion.}, the effect of that on the previous calculation is that the leftmost path of cells in the central row --- which has type 0 --- is no longer present. Therefore, this instance has no effect on the exponent of 2 resulting from the application of the complementation theorem. It follows that, regardless of how the $s$ holes are positioned, we get
\begin{equation}
\M(H)=2^{(m+1)-(s+1)}\M(H').
\label{ede}
\end{equation}
For the subsequent applications of the complementation theorem, let us again first analyze the situation in the case when none of the holes in $AR''_{m,m-a}\setminus O_{0,a_1}\cup\cdots\cup O_{0,a_s}$ communicates with the exterior.
Then the total number of rows containing cells decreases by one unit, and the number of central rows broken into $s+1$ paths of cells increases by one unit. Thus we have two fewer paths of type 1, and $s-1$ more paths of type $-1$, amounting to an exponent of 2 which is $-2\cdot1+(s-1)\cdot(-1)=s+1$ units less than it was at the previous application of the complementation theorem.

\begin{figure}[t]
\centerline{
%\hfill
{\includegraphics[width=0.40\textwidth]{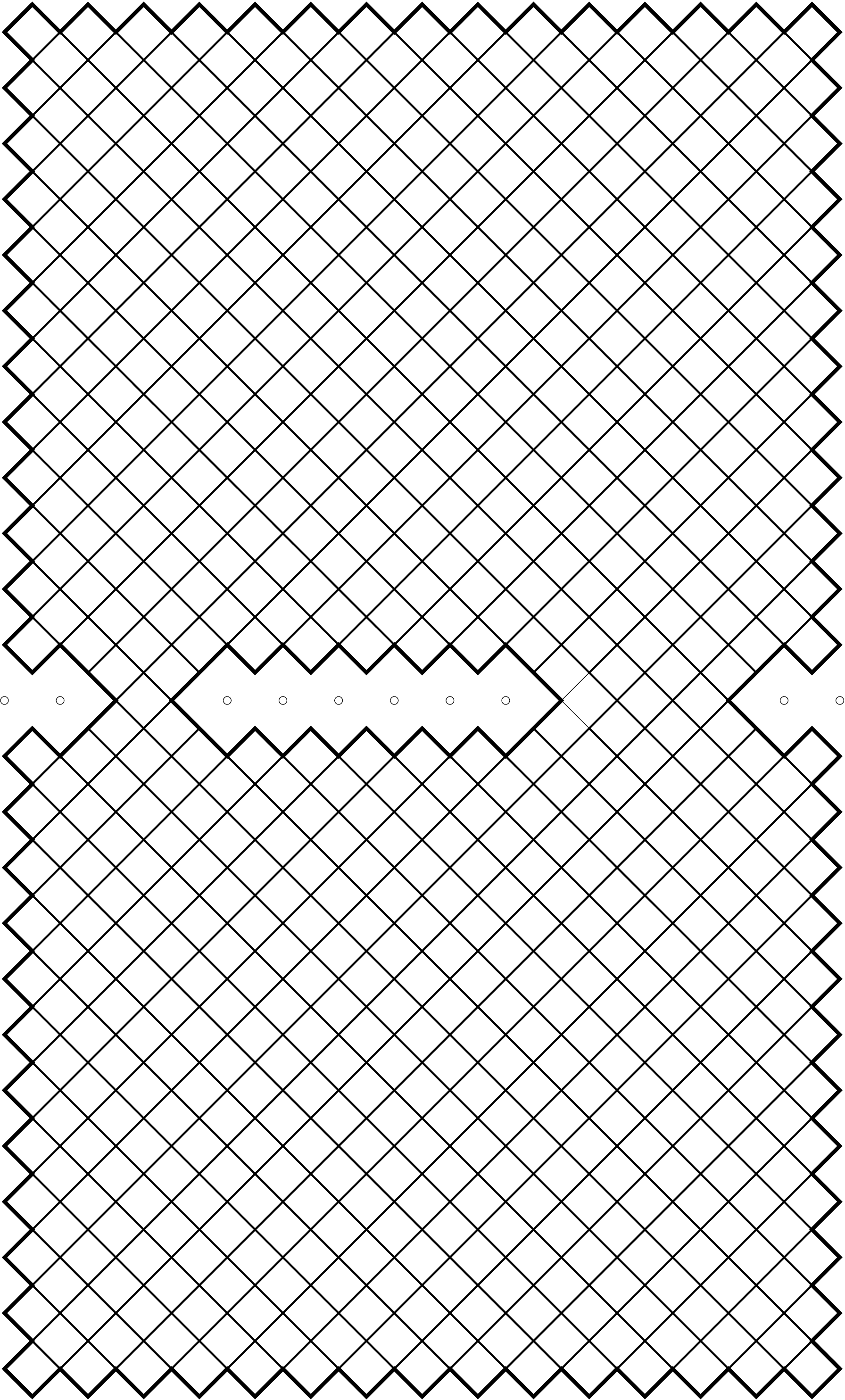}}
\hfill
{\includegraphics[width=0.42\textwidth]{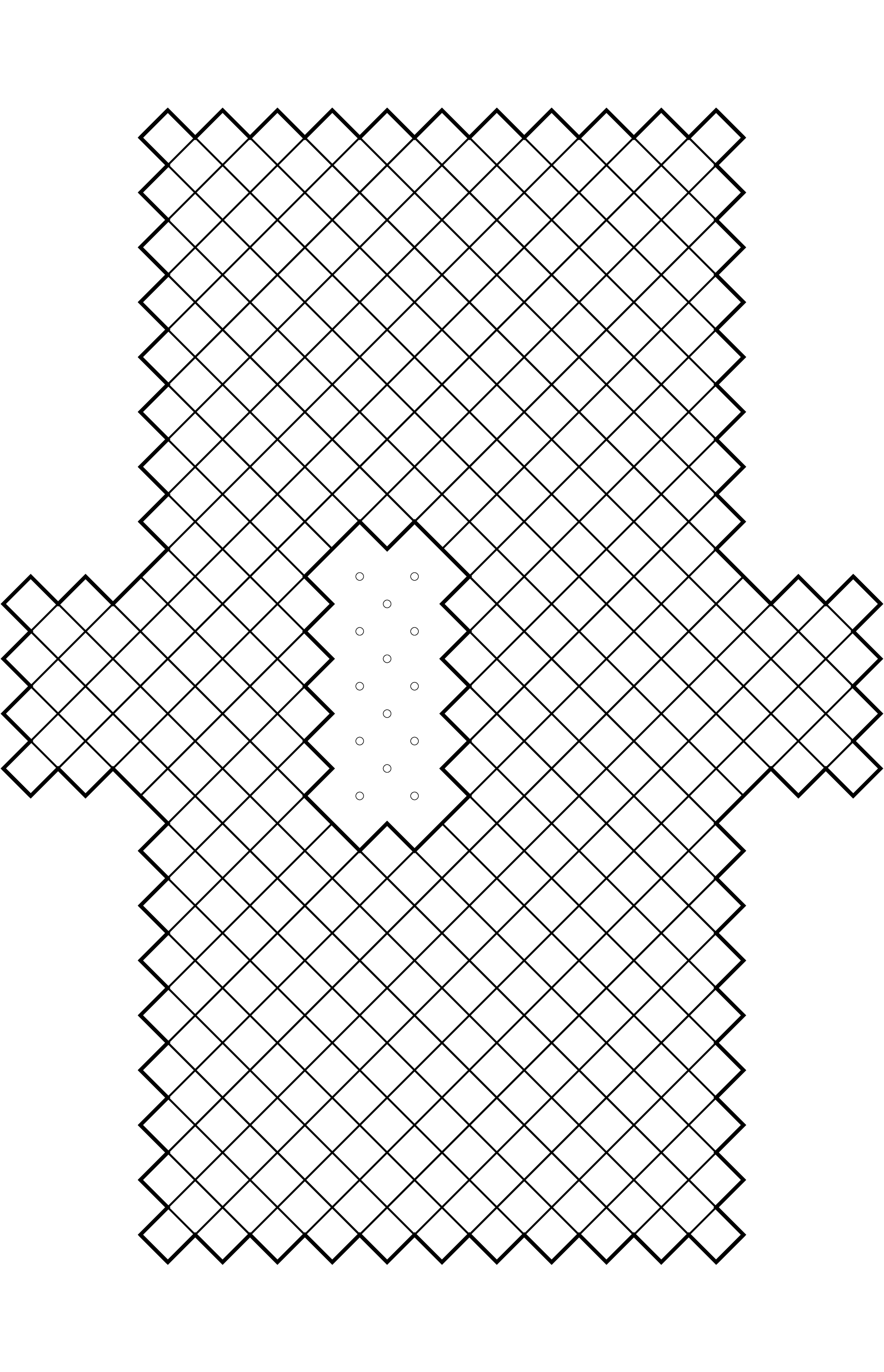}}
%\hfill
}
\vskip0.1in
\caption{ {\it Left.} An $R''$-type graph with holes grouped in consecutive runs, the first and last of which communicate with the exterior.  {\it Right.} After repeated applications of the complementation theorem to the graph on the left, we obtain the dual of the cruciform region with odd Aztec windows shown in Figure~\ref{fbd}.}
\vskip-0.1in
\label{fdx}
\end{figure}

Therefore, after $k$ successive application we obtain
\begin{align}
\M(AR''_{m,m-a}\setminus O_{0,a_1}\cup\cdots\cup O_{0,a_s})
&=
2^{\sum_{j=1}^k(m+1-j(s+1))}\M(AR''_{m-k,m-a-k}\setminus O_{k,a_1-k}\cup\cdots\cup O_{k,a_s-k})
\nonumber
\\
&=
2^{k(m+1)-{k+1\choose 2}(s+1)}\M(AR''_{m-k,m-a-k}\setminus O_{k,a_1-k}\cup\cdots\cup O_{k,a_s-k}),
\label{edf}
\end{align}
which proves \eqref{ebe} in this case.

To complete the proof, suppose for instance that both the leftmost and the rightmost holes in $AR''_{m,m-a}\setminus O_{0,a_1}\cup\cdots\cup O_{0,a_s}$ communicate with the exterior (such an instance is illustrated on the left in Figure \ref{fdx}). Note that in this case successive applications of the complementation theorem may lead to graphs whose appearance is quite different --- the graph in the picture on the right in Figure \ref{fdx} is the result of applying it successively four times to the graph shown on the left in the same figure (the outer boundary of this region belongs to the family of cruciform regions we studied in \cite{df}). However, we claim that what we found above, namely that with each successive application the exponent of 2 is reduced by $s+1$ units, is still true in this case.

Indeed, before the holes that communicate with the exterior ``turn inside out,'' this holds for the reason we mentioned above when we analyzed the first application of the complementation theorem: all that changes is that some paths of cells of type 0 are no longer present, and this clearly has no effect on the exponent of 2. What is left is to argue that the same is true after one or both extremal holes of $AR''_{m,m-a}\setminus O_{0,a_1}\cup\cdots\cup O_{0,a_s}$ have turned into protrusions. Suppose both have done so, and thus our graph has evolved into one that looks like the picture on the right in Figure \ref{fdx}. Then note that each of the central rows of 4-cycles consists of $s-1$ paths of cells, each having type $-1$ (see \cite{df} for details on the complementation theorem and how to work out the exponent of 2 when one applies it). This completes the argument.
$\hfill\square$

\medskip
{\it Proof of Theorem $\ref{tbd}$.} As it was the case for the proofs of Theorems \ref{tba} and \ref{tbb} (see the explanation in the first paragraph of the proof of Theorem \ref{tbb}), the outer boundary and the boundaries of the holes evolve in the same way in Theorems \ref{tbc} and \ref{tbd} under successive applications of the complementation theorem. Thus, the same calculations as in the proof of Theorem \ref{tbc} above prove that, for even $m$, the number of perfect matchings of $AR'''_{m,m-a}\setminus O_{0,a_1}\cup\cdots\cup O_{0,a_s}$ and $AR'''_{m-k,m-a-k}\setminus O_{k,a_1-k}\cup\cdots\cup O_{k,a_s-k}$ are related by precisely the same power of 2 as in \eqref{ebe}, thus proving \eqref{ebf}.
$\hfill\square$

\section{Proof of Theorem \ref{tca}}

The special case of the complementation theorem \cite[Theorem 2.1]{CT} we presented in Section 5 was what we needed for our proofs of Theorems \ref{tba}--\ref{tbd}. In order to prove Theorem \ref{tca} we need another special case of it. More precisely, let $H$ be a toroidal Aztec rectangle graph $T_{m,n}$ with holes in the shape of disjoint odd Aztec rectangles. Fix a chessboard coloring of the 4-cycles of $T_{m,n}$, and let $G$ be the graph consisting of the union of the shaded 4-cycles --- the cells of $G$ --- that contain at least one edge of $H$. As in the special case presented in Section 5, a path of cells of $G$ is a maximal set of contiguous cells lined up either horizontally or vertically along the grid diagonals of $T_{m,n}$; the difference is that now such a path of cells can wrap around the torus. 

In particular, it is possible to have paths of cells formed by a ring of cells --- such paths have no extremal vertices, and are defined to have type 0. If a path of cells is not a ring, the two vertices in it that are furthest apart (in the path of cells, regarded as a graph) are called extremal; the type of such a path of cells is defined as one less than the number of its extremal vertices contained in $H$. The complement $H'$ of $H$ with respect to its cellular completion $G$ is defined as before: the vertex set of $H'$ is obtained from that of $H$ by discarding the extremal vertices contained in $H$, and including those not contained in $H$ (a cellular completion of the graph in Figure \ref{fca} is shown on the left in Figure \ref{fea}; its complement with respect to the indicated cellular completion is shown on the right in the same figure). Then  by \cite[Theorem 2.1]{CT}, equation \eqref{eda} still holds (as before, $L_1,\dotsc,L_k$ is a partition of the set of cells of $G$ into disjoint paths).

\begin{figure}[t]
\centerline{
%\hfill
{\includegraphics[width=0.46\textwidth]{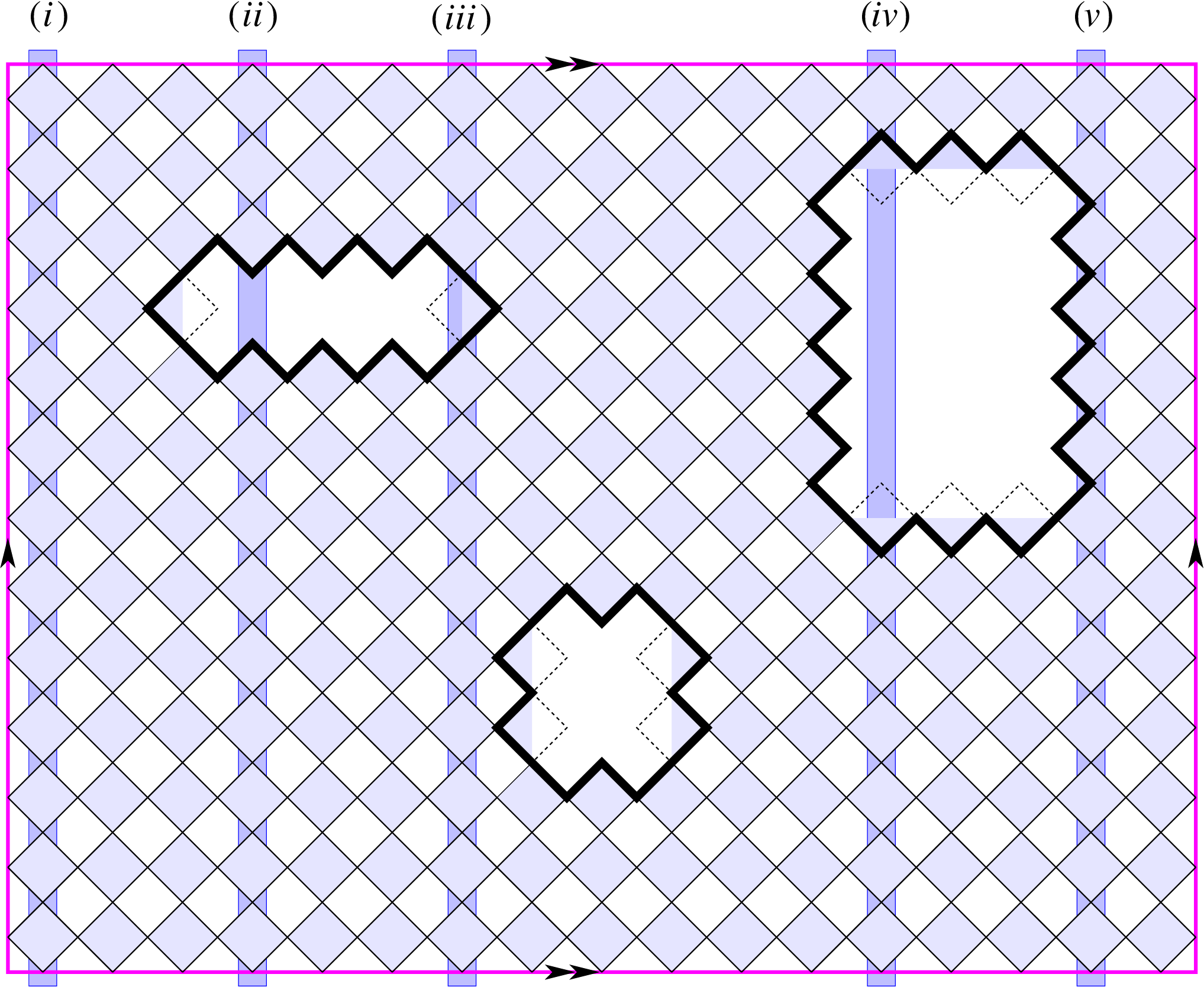}}
\hfill
{\includegraphics[width=0.46\textwidth]{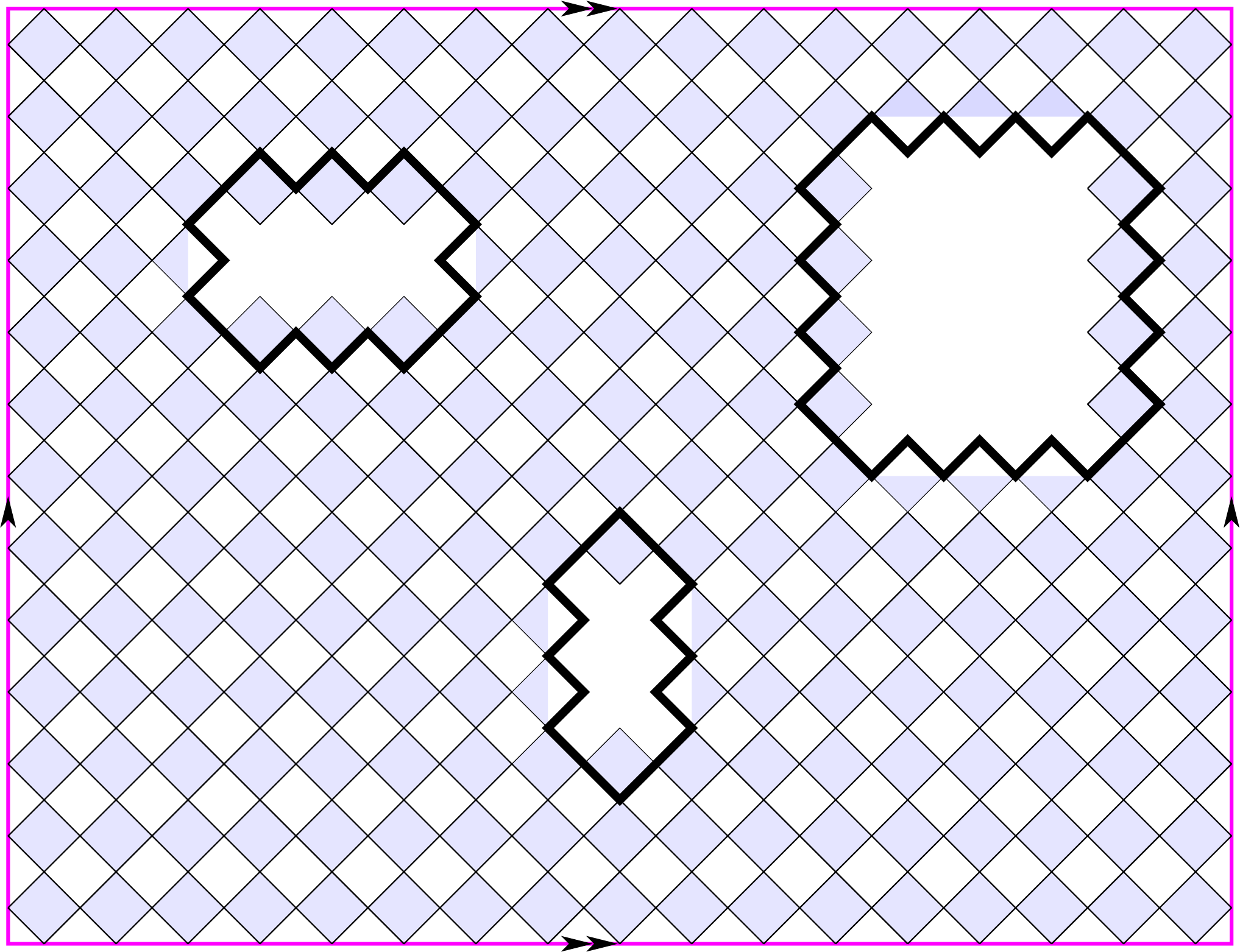}}
%\hfill
}
\vskip0.1in
\caption{ {\it Left.} A cellular completion of the graph in Figure \ref{fca}, with five types of vertical paths of cells; the thick black contours indicate the shapes of the original holes.  {\it Right.} The complement of the the graph in Figure \ref{fca} with respect to the cellular completion on the left; the thick black contours indicate the shapes of the evolved forms of the original holes.}
\vskip-0.1in
\label{fea}
\end{figure}

\medskip
 {\it Proof of Theorem $\ref{tca}$.} By Lemma \ref{tcaa} it follows that the evolved forms $\e(O_1),\dotsc,\e(O_r)$ of the holes are also mutually disjoint.

Fix a chessboard coloring of the 4-cycles of the toroidal Aztec rectangle graph $T_{m,n}$. Let $H$ be the graph $T_{m,n}\setminus O_1\cup\cdots\cup O_r$. Let $G$ be the union of the 4-cycles of $T_{m,n}$ that contain at least one edge of $H$. Then $G$ is a cellular graph, and it is a cellular completion of $H$. 

We first prove the case when each vertical ring of 4-cycles of $T_{m,n}$ meets at most one of the holes $O_1,\dotsc,O_r$. There are then five kinds of vertical paths of cells (one of each type is indicated on the left in Figure \ref{fea}): 

\medskip
$(i)$ vertical rings, which have type 0, 

$(ii)$ vertical paths crossing through a white-placed odd Aztec rectangle hole, for which both extremal vertices are contained in $H$, so have type 1,  

$(iii)$ vertical paths just grazing into a white-placed odd Aztec rectangle hole, which are rings (even though they contain some partial cells), so have type 0, 

$(iv)$ vertical paths crossing through a black-placed odd Aztec rectangle hole, for which none of the extremal vertices are contained in $H$, so have type $-1$, 

$(v)$ vertical paths just touching a white-placed odd Aztec rectangle hole, which are rings, so have type 0.

\medskip
When we construct the complement $H'$, since there is no ``outer boundary'' (the vertices on the boundary of $AR_{m,n}$ were identified to produce $T_{m,n}$), the only change compared to $H$ will be in the evolution of the boundary of the holes. We already know, from Section 5, that a white-placed hole $O$ in the shape of $O_{k,l}$ will evolve so as to acquire the shape $O_{k+1,l-1}=\e(O)$. One readily checks that a black-placed hole $O'$ in the shape of $O_{k,l}$ will evolve to get shape $O_{k-1,l+1}=\e(O')$. Therefore, except for the exponent of 2, it is clear that the complementation theorem yields equation \eqref{eca}.

To get the exponent of 2, note that only paths of cells of type $(ii)$ and $(iv)$ contribute to it. It is easy to see that for each white-placed hole of shape $O_{k,l}$ there are $l$ paths of type $(ii)$, and for each black-placed hole of shape $O_{k,l}$ there are $l+1$ paths of type $(ii)$. This proves equation \eqref{eca} in this case.

%For the general case, the same reasoning can be applied. The difference is that, as we travel along a vertical ring of 4-cycles of $T_{m,n}$, we may encounter more than one hole. However, encountering for instance one crossing through the interior of a white-placed hole and then another crossing of the same type, results in two paths of cells, both of type 1; so effectively the counting analysis we made above still holds. The same happens for all possible combinations of successive crossings . This completes the proof.

Now consider the effect of including an additional hole $O$ that has cells in vertical rings of 4-cycles of $T_{m,n}$ that go through some of the pre-existing holes. For definiteness, assume the new hole $O$ has the shape of the odd Aztec rectangle $O_{k,l}$ and is white-placed. We need to understand the effect on the sum of the types of all vertical paths of cells (which is the exponent of 2 in the complementation theorem) caused by the addition of this new hole. We claim that the effect is that this sum of types increases by $l$ (which, recall, is precisely the flank charge~$\f(O)$ of $O$).

Indeed, the only vertical paths of cells affected by the inclusion of $O$ are those contained in the $l+2$ vertical rings of 4-cycles that meet $O$. For two of these rings (the leftmost and rightmost of the $l+2$), the paths of cells contained in them remain the same (this is because all changes are of the type ``a full cell $c$ is replaced by a partial cell containing both the top and the bottom vertices of $c$,'' and thus no path of cells is interrupted by the change). For each of the middle $l$ vertical rings, the inclusion of $O$ causes one path of cells to be interrupted, and the two newly created ends of paths are both full cells. It is easy to check that irrespective of the type of the interrupted path of cells, the sum of the types of the two newly created paths is 1 more that the type of the original path: if the latter had type 1, both new paths have type 1; if it had type 0, one of the new paths has type 0, the other type 1; and if it had type $-1$, both new paths have type 0.

Therefore, the exponent of 2 increases by $\f(O)$ by the inclusion of the white-placed odd Aztec window $O$. A similar argument shows that the same holds when $O$ is black-placed. This completes the proof.
$\hfill\square$

\section{Concluding remarks}

In this paper we presented an Aztec-diamond-like shape for a hole which, when included in Aztec-diamond-like regions, leads to regions whose number of tilings are given by simple product formulas, thus providing a ``round counterpart'' of Propp's Problem 19 in \cite{ProppList}. More precisely, the families of regions we found are Aztec rectangles, in which holes in the shape of odd Aztec rectangles of a common height have been placed symmetrically about the horizontal symmetry axis $\ell$ (or about the translation of $\ell$ one unit southeast); see Theorems \ref{tba}--\ref{tbd} and Figures \ref{fba}--\ref{fbc}. In some instances our regions look in fact different, being cruciform regions with holes (see Remark 4 and Figures \ref{fbd} and \ref{fdx}).

We also proved that the number of perfect matchings of toroidal Aztec rectangle graphs with holes in the shape of odd Aztec rectangles get multiplied by an explicit power of 2 as the shape of the holes is evolved in a natural way (see Theorem \ref{tca}). In particular, we proved a curious symmetry of the correlation of diagonal slits, which already holds at the level of the finite size correlation (see Theorem \ref{tcd} and Figure \ref{fcb}). It would be interesting to have a direct proof for the fact that the correlation of diagonal slits is invariant under 90-degree rotations.

\end{document}